\documentclass[10pt]{article}
\usepackage{amsmath}
\usepackage{amssymb}
\usepackage{amsthm}
\usepackage{graphicx}
\usepackage{verbatim}
\usepackage{enumerate}
\usepackage{color}
\usepackage{xcolor}
\definecolor{cgreen}{rgb}{0.0,0.42.0.24}
\definecolor{cpurple}{rgb}{0.78,0,0.82}
\usepackage{mathrsfs}
\usepackage[colorlinks=true,linkcolor=blue,citecolor=blue,pdfpagelabels=false]{hyperref}
\usepackage{longtable}
\numberwithin{equation}{section}
\usepackage[english]{babel}
\usepackage{geometry,amsmath,amsthm,amssymb,latexsym,graphicx}
\usepackage{tikz}

\newcommand{\eqb}{\begin{equation}}
\newcommand{\eqe}{\end{equation}}

\theoremstyle{plain}
\newtheorem{thm}{Theorem}[section]

\newtheorem{lem}[thm]{Lemma}
\newtheorem{pr}[thm]{Proposition}
\newtheorem{cor}[thm]{Corollary}
\newtheorem{defn}[thm]{Definition}

\newtheorem{example}[thm]{Example}
\newtheorem{problem}{Problem}

\theoremstyle{remark}

\newtheorem*{unremark}{Remark}

\vfill

\allowdisplaybreaks

\font\elevenss=cmss11

\font\eightss=cmss8

\font\sixss=cmss8 at 6pt

\newfam\ssfam
\textfont\ssfam=\elevenss \scriptfont\ssfam=\eightss
 \scriptscriptfont\ssfam=\sixss
\def\ss{\fam\ssfam \elevenss}%

\def\N{\mathbb{N}}

\def\R{\mathbb{R}}

\def\C{\mathbb{C}}

\def\CC{{\cal C}}

\def\ee{\varepsilon}

\def\Cox{\hfill \Box}

\def\one{{\bf 1}}
\def\|{{\, | \, }}

\def\xx{{\bf x}}
\def\yy{{\bf y}}
\def\rr{{\bf r}}
\def\zz{{\bf z}}

\def\ww{{\bf w}}
\def\vv{{\bf v}}
\def\uu{{\bf u}}
\def\mm{{\bf m}}
\def\dd{\mathbf{d}}

\def\M{{\cal M}}

\def\cD{{\cal D}}
\def\D{{\cal D}}

\def\nbd{{\mathcal N}}
\def\grad{{\nabla}}

\def\slab{\mathtt{slab}}
\def\bottom{\mathtt{bottom}}
\def\zero{{\bf 0}}
\def\sing{{\mathcal V}}
\def\TT{{\bf T}}
\def\Ll{\mathbf{L}}
\def\rhat{{\hat{\rr}}}
\def\shat{\hat{\bf s}}

\def\Res{{\rm Res}\,}

\def\nbdhat{{\widehat{\nbd}}}
\def\amoeba{\mbox{\ss amoeba}}

\def\I{{\mathcal I}}
\def\homot{{\mathbf \Phi}}
\def\e{{\bf e}}
\def\singq{{\tilde{\sing}}}
\def\hq{{\tilde{h}}}
\def\hyper{{\mathtt H}}
\def\sph{\mathtt{S}}
\def\hyperr{{\mathcal H'}}

\def\sphh{{\mathcal S'}}
\def\north{{\mathfrak n}}
\def\nint{{\mathfrak m}}
\def\xt{{\bf x}}
\def\yt{{\bf y}}
\def\zt{{\bf z}}
\def\sign{{\rm sgn}\,}

\def\bH{{\bf H}}

\def\locus{\Sigma}

\def\leray{\mathtt{o}}
\def\disk{\bullet}
\def\chain{{\gamma}}
\def\Comp{\mathbb{C}}
\def\Real{\mathbb{R}}
\def\Int{\mathbb{Z}}
\def\WW{{\mathbf W}}
\def\bpt{\Xi}

\begin{document}

\begin{titlepage}

\begin{center}
{\LARGE \bf Asymptotics of multivariate sequences in the presence of a lacuna}
\footnote{SM and RP gratefully acknowledge the support and 
hospitality of the Erwin Schr{\"o}dinger Institute}
\end{center}

{\large

\noindent{\bf Abstract:}
We explain a discontinuous drop in the exponential growth rate for 
certain multivariate generating functions at a critical parameter
value, in even dimensions $d \geq 4$.  This result depends on 
computations in the homology of the algebraic variety where the 
generating function has a pole.  These computations are similar to, 
and inspired by, a thread of research in applications of complex 
algebraic geometry to hyperbolic PDEs, going back to Leray, Petrowski, 
Atiyah, Bott and G\"arding.  As a consequence, we give a topological 
explanation for certain asymptotic phenomenon appearing in the
combinatorics and number theory literature.  Furthermore, we show
how to combine topological methods with symbolic algebraic computation 
to determine explicitly the dominant asymptotics for such 
multivariate generating functions, giving a significant
new tool to attack the so-called connection problem for
asymptotics of P-recursive sequences.  This in turn enables the 
rigorous determination of integer coefficients in the Morse-Smale 
complex, which are difficult to determine using direct geometric
methods.
\vfill
\vfill

{\sc \small
Yuliy Baryshnikov,
University of Illinois,
Department of Mathematics,
273 Altgeld Hall 1409 W. Green Street (MC-382), 
Urbana, IL 61801, {\tt ymb@illinois.edu},
partially supported by NSF grant DMS-1622370. 

Stephen Melczer,
University of Waterloo,
Department of Combinatorics and Optimization,
200 University Avenue,
Waterloo, ON N2L 3G1, {\tt smelczer@uwaterloo.ca},
partially supported by an NSERC postdoctoral fellowship
and NSERC Discovery Grant. 

Robin Pemantle,
University of Pennsylvania,
Department of Mathematics,
209 South 33rd Street,
Philadelphia, PA 19104, {\tt pemantle@math.upenn.edu},
partially supported by NSF grant DMS-1612674. 
}


}
\end{titlepage}

\section{Introduction}

Let $k \geq 1$ be an integer, and for $P$ and $Q$ coprime polynomials 
over the complex numbers let
\begin{equation} \label{eq:F}
F(\zz) = \frac{P(\zz)}{Q(\zz)^k} 
   = \sum_{\rr\in\mathbb{Z}^d} a_\rr \zz^\rr = \sum_{\rr\in\mathbb{Z}^d} a_\rr z_1^{r_1} \cdots z_d^{r_d}
\end{equation}
be a rational Laurent series converging in some open domain
$\cD \subset \mathbb{C}^d$.  The field of Analytic Combinatorics 
in Several Variables
(ACSV) describes the asymptotic determination of the coefficients 
$a_\rr$ via complex analytic methods.  Let $\sing = \sing_Q$ denote 
the algebraic set $\{ \zz : Q(\zz) = 0 \}$ containing the singularities
of $F(\zz)$.  The methods of ACSV, summarized below,
vary in complexity depending on the nature of $\sing$.  When $\sing$ 
is a smooth manifold, for instance when $Q$ and $\grad Q$ do not vanish simultaneously,
explicit formulae may be obtained that are universal outside of cases
when the curvature of $\sing$ vanishes~\cite{PW1,PW-book}.  When $\sing$ is the
union of transversely intersecting smooth surfaces, similar residue
formulae hold~\cite{PW2,PW-book,BMP-morse}.  The next most difficult case
is when $\sing$ has an isolated singularity whose tangent cone is 
quadratic, locally of the form $x_1^2 - \sum_{j=2}^d x_j^2$.  
These points, satisfying the 
{\em cone point hypotheses}~\cite[Hypotheses~3.1]{BP-cones}, are called 
{\em cone points}; the necessary complex 
analysis in the vicinity of a cone point singularity, based on 
the work of~\cite{ABG}, is carried out in~\cite{BP-cones}.  

Let $|\rr|=|r_1|+\cdots+|r_d|$. In each of these cases, asymptotics may be found of the form 
\begin{equation} \label{eq:form}
a_\rr \sim C(\rhat) |\rr|^\beta \zz_* (\rhat)^{-\rr}
\end{equation}
where $C$ and $\zz_*$ depend continuously on the direction 
$\rhat := \rr / |\rr|$.  A very brief summary of the methodology 
is as follows.  The multivariate Cauchy integral formula gives
\begin{equation} \label{eq:cauchy}
a_\rr = \left ( \frac{1}{2 \pi i} \right )^d \int_T \zz^{-\rr} 
   F(\zz) \frac{d\zz}{\zz},
\end{equation}
where $T \subseteq \cD$ is a torus in the domain of convergence
and $d\zz / \zz$ is the logarithmic holomorphic volume form
$z_1^{-1} \cdots z_d^{-1} dz_1 \wedge \cdots \wedge dz_d$.
Expand the chain of integration $T$ so that it passes through the 
variety $\sing$, touching it for the first time at a point $\zz_*$ 
where the logarithmic gradient of $Q$ is normal to $\sing$, and  
continuing to at least a multiple $(1+\ee)$ times this polyradius.
Let $\I$ be the intersection with $\sing$ swept out by the homotopy
of the expanding torus.  The residue theorem, described 
in Definition~\ref{def:residue} below, says that the integral~\eqref{eq:cauchy} is
equal to the integral over the expanded torus plus the integral
of a certain residue form over $\I$.  Typically, $\zz^{-\rr}$ is
maximized over $\I$ at $\zz_*$, and integrating over $\I$ yields
asymptotics of the form~\eqref{eq:form}.   

In the case of an isolated singularity with quadratic tangent cone,
Theorem~3.7 of~\cite{BP-cones} gives such a formula but excludes the 
case where $d = 2m > 2k+1$ is an even integer and $d-1$ 
is greater than twice the power $k$ in the denominator of~\eqref{eq:F}.
In that paper the asymptotic estimate obtained is only 
$a_\rr = o(|\rr|^{-m} \zz_*^{-\rr})$ for all $m$, due to the 
vanishing of a certain Fourier transform.
This leaves open the question of what the correct asymptotics are,
and whether they are smaller by a factor exponential in $|\rr|$.

In~\cite{BMPS} it is shown via diagonal extraction that, for $k=1$ and 
a class of polynomials $Q$ with an isolated quadratic cone singularity,
in fact $a_{n, \ldots , n}$ has strictly smaller exponential order
than expected.  Diagonal extraction applies only
to coefficients of monomials precisely on the diagonal -- i.e., where every
variable has the same power -- leaving open the question 
of behavior in a neighborhood of the diagonal\footnote{To see why this is
important, consider the function $(x-y)/(1+x+y)$ that generates differences
of binomial coefficients ${i+j+1 \choose i} - {i+j+1 \choose j}$. The diagonal
coefficients where $i=j=n$ are zero but the growth of those nearby
approaches a constant times $n^{-1/2} 4^n$.}, 
and leaving open the question of whether this behavior holds 
beyond the particular class, for all polynomials with cone point
singularities.  The purpose of the present paper is to use ACSV methods
to show that indeed the behavior is universal for cone points,
to prove that it holds in a neighborhood of the diagonal, and to give a 
topological explanation.

\section{Main results and outline} \label{sec:results}

\subsection*{Main result}

Let $F,P,Q$ and $\{ a_\rr \}$ be as in~\eqref{eq:F}, choosing signs  
so that $Q(\zero) > 0$.  Throughout the paper we denote by $\Ll : \C_*^d
\to \R^d$ the coordinatewise log-modulus map
\begin{equation} \label{eq:Ll}
\Ll (\zz) := \log |\zz| 
   = \left ( \log |z_1| , \ldots , \log |z_d| \right ) \, .
\end{equation}
Let $\C_* := \C \setminus \{ 0 \}$ and
let $\M := \C_*^d \setminus \sing$ be the domain of holomorphy of 
$\zz^{-\rr} F(\zz)$ for sufficiently large $\rr$.  Let $\amoeba$
denote the amoeba of $Q$, defined by $\amoeba := \{ \Ll(\zz) :
\zz \in \sing \}$.  It is known~\cite{GKZ} that the components of
the complement of the amoeba are convex and correspond to
Laurent series expansions for $F$, each component being a 
logarithmic domain of convergence for one series expansion.
Let $B$ denote the component of $\amoeba^c$ such that the given series 
$\sum_\rr a_\rr \zz^\rr$ converges whenever $\zz = \exp (\xx + i \yy)$ 
with $\xx \in B$. 

We refer to the torus $T(\xx):=\Ll^{-1}(\xx)$ 
as the {\bf torus over} $\xx$.
For any $\rr \in \R^d$ we denote $\rhat := \rr / |\rr|$ and 
$$h_{\rhat} := - \sum_{j=1}^d \hat{r}_j \log |z_j| \, .$$
For a subset $A \subset \C^d$, when $\rhat$ and $\zz_*$ are understood,
we use the shorthand
\begin{equation} \label{eq:filtration}
A({-\ee}) := A \cap \{\zz : h_{\rhat} (\zz) < h_{\rhat} (\zz_*) - \ee \}
   \, .
\end{equation}

Assume that $\sing$ intersects the torus
$\{ \exp (\xx_* + i \yy): \yy \in (\R / (2\pi))^d \}$ at the unique 
point $\zz_* = \exp (\xx_*)$.  We will be dealing with the situation 
where $\sing$ has a quadratic singularity at $\zz_*$.  More specifically, 
we will assume that $Q$ has a real hyperbolic singularity at $\zz_*$.
\begin{defn}[quadric singularity]
\label{def:quadratic} 
We say that $Q$ has a {\bf real hyperbolic quadratic singularity} at 
$\zz_*$ if $Q(\zz_*)=0$, the gradient $\nabla Q(\zz_*)=0$, 
and the quadratic part $q_2$ of 
$Q=q_2(\zz)+q_3(\zz)+\ldots$ at $\zz_*$ is a real quadratic form of 
signature $(1,d-1)$; in other words, there exists a real linear coordinate 
change so that $q_2(\uu) = u_d^2 - \sum_{j=1}^{d-1} u_j^2 + O(|\uu|^3)$
for $\uu$ a local coordinate centered at $\zz_*$.
\end{defn}

We denote by $T_{\xx_*} (B)$ the open tangent cone in $\R^d$ to 
the component $B$ of $\amoeba (Q)^c$, namely all vectors
$\vv$ at $\xx_* := \Ll (\zz_*)$ such that $\xx_* + \ee \vv \in B$
for sufficiently small $\ee$.  The inequality defining $T_{\xx_*} (B)$
is the same as the inequality $\widetilde{Q} (\vv) > 0$ where $\widetilde{Q}$
is the leading (homogeneous quadratic) term of 
$Q (\exp (\xx_* + \vv + i \yy_*))$, along with an inequality 
specifying $T_{\xx_*} (B)$ rather than $- T_{\xx_*} (B)$.

\begin{defn}[tangent cone; supporting vector]
\label{def:tangent_cone}
The vector $\rr$ is said to be {\bf supporting} at $\zz_*$ if
$h_{\rr}$ attains its maximum on the closure of $B$ at $\xx_*$
and if $\{dh_\rr = 0 \}$ intersects the tangent cone $T_{\xx_*}(B)$
only at the origin.  The open convex cone of supporting vectors is
denoted $\nbd$ and the set of unit vectors over which it is a cone
is denoted $\nbdhat$.
\end{defn}

\begin{thm}[main theorem] \label{th:lacuna}
Let $P$ be holomorphic in $\Comp^d$, $Q$ a Laurent polynomial, $k$ a 
nonnegative integer, $B$ a component in the complement of the amoeba
of $Q$, and $\sum_{\rr \in E} a_\rr \zz^\rr$ the corresponding 
Laurent series expansion of $F = P/Q^k$. 

Suppose that $Q$ has real hyperbolic quadratic 
singularity at $\zz_* = \exp(\xx_*)$ such that $\xx_*$ belongs to
the boundary of $B$ and $\zz_*$ is the unique 
intersection of the torus $\TT(\xx_*)$ with $\sing$.

Let $K \subseteq \nbdhat$ be a compact set
and suppose that $d$ is even and $2k<d$.
\begin{enumerate}[(i)]
\item If $\ee>0$ is small enough then for any $\rhat\in K$ there exists 
a compact cycle $\Gamma(\rhat)$, of volume uniformly bounded in $\rhat$, such that for all $\rhat \in K$ the cycle 
$\Gamma(\rhat)$ is supported on $\M({-\ee})$ and
\begin{equation} \label{eq:exact}
a_\rr = \int_{\Gamma(\rhat)} \zz^{-\rr} \frac{P}{Q^k} \frac{d\zz}{\zz}.
\end{equation}
\item
If $P$ is a polynomial, then
\begin{equation} \label{eq:int chain}
a_\rr = \int_{\gamma(\rhat)} \Res_\sing \zz^{-\rr} \frac{P}{Q^k} \frac{d\zz}{\zz} 
\end{equation}
for all but finitely many $\rr \in E$, where 
$\gamma(\rhat)$ is a compact $(d-1)$ cycle in $\sing(-\ee)$ 
of volume uniformly bounded in $\rhat$ and $\Res_\sing$ is 
the residue operator defined below in Section \ref{sec:preliminaries}.
\end{enumerate}
\end{thm}

The heuristic meaning of this result is that, for purposes of computing the Cauchy integral,
the chain of integration in~\eqref{eq:cauchy} can be slipped below 
the height $h_\rr(\zz_*)$ of the singular point $\zz_*$.

\subsection*{Motivation}

Our motivating example for this theorem comes from the 
Gillis-Reznick-Zeilberger family of generating functions~\cite{laguerre},
which has origins in the work of Szeg{\"o}~\cite{Szego1933} and 
Askey and Gasper~\cite{AskeyGasper1977} and is further
discussed in~\cite[Theorems~9~--~12]{BMPS}.
\begin{example}[GRZ function at criticality] \label{eg:GRZ}
In four variables, let $F_\lambda(\zz) := 1 / (1 - z_1 - z_2 - z_3 - z_4 
+ \lambda z_1 z_2 z_3 z_4)$ and, for convenience, write $F(\zz)=F_{27}(\zz)$.
It is shown in~\cite{BMPS} via ACSV results for smooth functions
that the diagonal exponential growth rate $|a_{n,n,n,n}|^{1/n}$ of the
power series coefficients of $F_\lambda$
is a function of $\lambda$ that approaches~81 as $\lambda \to 27$.
At the critical value~27, however, the denominator $Q$ of $F$ has a 
real hyperbolic quadric singularity at $\zz_* := (1/3, 1/3, 1/3, 1/3)$.  
Theorem~\ref{th:lacuna} has the immediate consequence that
the exponential growth of $a_\rr$ for $\rhat$ in a neighborhood 
of the diagonal is strictly less than that of 
$\zz_*^{-\rr}=81^{|\rr|}$.  Thus there is a drop in the exponential rate
at criticality.
\end{example}

With a little further work, understanding of the drop can be sharpened 
considerably.  In Section~\ref{sec:GRZ} we state a result for
general functions satisfying the conditions of Theorem~\ref{th:lacuna}.
The result, Theorem~\ref{th:exp}, sharpens Theorem~\ref{th:lacuna},  
quantifying the exponential drop by pushing the contour $\Gamma$ 
down all the way to the next critical value.  It is a direct consequence
of Theorem~\ref{th:lacuna} together with a deformation result
of~\cite{BMP-morse}.  In the case of the GKZ function 
at criticality, the following explicit asymptotics result.

\begin{thm} \label{th:zeta}
The diagonal coefficients $a_{n,n,n,n}$ in the power series 
expansion of $F(\zz)$ from Example~\ref{eg:GRZ} 
have an asymptotic expansion in decreasing
powers of $n$, beginning
\begin{eqnarray}
a_{n,n,n,n} & = & 3 \cdot \left ( 
   \frac{\left(4i\sqrt{2}-7 \right)^n}{n^{3/2}} 
   \, \frac{\left(5i-\sqrt{2}\right)\sqrt{-2i\sqrt{2}-8}}{24\pi^{3/2}} 
   + \frac{\left(-4i\sqrt{2}-7 \right)^n}{n^{3/2}} 
   \, \frac{\left(-5i-\sqrt{2}\right)\sqrt{2i\sqrt{2}-8}}{24\pi^{3/2}} 
   \right ) \label{eq:anLambda}\\[2ex]
&& + O\left( 9^n \, n^{-5/2} \right). \nonumber
\end{eqnarray}
More generally, as $\rr \to \infty$ and $\rhat$ varies over 
some compact neighborhood of the diagonal, there is a uniform estimate
$$a_\rr = p_\rhat^n \, n^{-3/2} \cdot C_{\rhat} \cos (n \alpha_{\rhat} 
   + \beta_{\rhat}) + O(p_\rhat^n n^{-5/3}) \,, $$
where $p_\rhat,C_\rhat,\alpha_\rhat,$ and $\beta_\rhat$ vary continuously 
with $\rhat$ and specialize to produce~\eqref{eq:anLambda}
when $\rhat$ is on the diagonal.
\end{thm}

\subsection*{Heuristic argument}

Our plan is to expand a torus $\TT$ of integration representing 
series coefficients via the Cauchy integral theorem using
a homotopy $\bH$ that takes 
it through the point $\zz_*$ and beyond.  Let 
$\sing_*=\sing \cap \C_*^d$ denote the points of $\sing$ with no coordinate
vanishing.  A classical construction, due to Leray, Thom and others,
shows that $\TT$ is homologous in $H_d (\M)$ to a cycle $\Gamma'$
which coincides above height $h(\zz_*) - \ee$ with a tube around a
cycle $\sigma$; the height $h_\rr$ is maximized on $\sigma$ at the
point $\zz_*$ and the chain $\sigma$ is the intersection of $\bH$ with
$\sing_*$.  We would like to see that $\sigma$ is homologous to 
a class supported on $\sing(-\ee)$.  

To do this, we compute the intersection $\sigma$ directly in coordinates 
suggested by the hypotheses of the theorem. In particular, we use local 
coordinates where, after taking logarithms, $\sing$ is the cone 
$\{ z_1^2 - \sum_{j=2}^d z_j^2 = 0 \}$, and select a homotopy $\bH$ from 
$\xx + i \R^d$ to $\xx' + i \R^d$ with $\xx \in B$ so that the line 
segment $\overline{\xx \xx'}$ is perpendicularly bisected by the support
hyperplane to $B$ at $\xx_*$.  In these coordinates, the intersection
class $\I$ is the cone $\{ i \yy : \yy \in \R^d \mbox{ and } 
y_1^2 = \sum_{j=2}^d y_j^2 \}$.  The residue is singular at
the origin (in new coordinates) but converges when $d > 2k + 1$.
Inside the variety $\sing$, the cone $\I$ may be folded down so 
as to double cover the cone $\{ x + i \yy : \yy \in \R^d, x > 0, 
y_1 = 0 \mbox{ and } x^2 = |\yy|^2 \}$, as shown in Figure~\ref{fig:fold}.
The two covering maps have opposite orientations when $d$ is even.  
The critical points of $h_{\rr}$ restricted to $\sing$ are obstructions 
for deforming the contour of integration downwards, and in this case 
the residue integral vanishes and the contour may be further deformed 
until it encounters the next highest critical point.

\begin{figure}[!ht]
\centering
\includegraphics[width=3in]{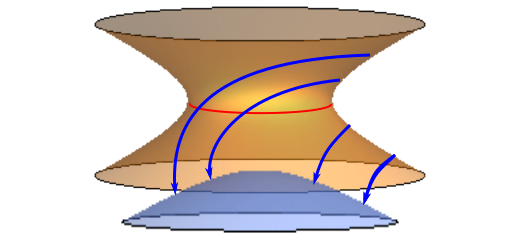}
\caption{folding the cone down in two opposite directions}
\label{fig:fold}
\end{figure}

\subsection*{Outline of actual proof}

The proof cannot precisely follow the heuristic argument because
the intersection cycle construction and the residue integral 
theorem work only when $\sing_*$ is smooth. 
We remark that the same trouble arose in the setting of 
\cite{BP-cones}. There, the authors adopt the method of \cite{ABG} 
to reduce the local integration cycle to its projectivized,
compact counterpart: the so-called Petrowski or Leray cycles. 
That path required significant investment into analytic auxiliary 
results and, more importantly, would not immediately prove that 
the integration cycle in the presence of lacuna (i.e., when $d$ 
is even and the denominator degree not too high) allows one to 
``slide'' the integration cycle below the height of the cone point. 

Thus we use a different strategy, first perturbing the denominator 
so that the perturbed varieties on which $Q(\zz) = c$ for small 
$c$ become smooth. This kind of regularization also has the advantage, 
compared to what was used in \cite{BP-cones}, that we obtain 
information about the behavior of coefficients of the generating 
functions $P/(Q-c)^k$. 
Next, we study the behavior of the coefficients 
of the resulting generating functions as $c\to 0$. We
denote the zero set of $Q(\zz)-c$ by $\sing_c$, write 
$(\sing_c)_*$ for the points of $\sing_c$ with non-zero coordinates, 
and denote the restriction of $\sing_c$ to its points of 
height at most $h(\zz_*) - \ee$ by $\sing_c (\leq -\ee)$.
It is easiest to work in the lower dimensional setting, with 
$\sigma_c$ on $(\sing_c)_*$ rather than $\Gamma_c$ on $\M_c$, 
and to work in relative homology of $(\sing_c)_*$ with respect to 
$\sing_c (\leq - \ee)$.

Section~\ref{sec:quadric} lays the theoretical groundwork 
by computing the explicit intersection cycle in a limiting case of
the perturbed variety as $c \downarrow 0$; this is a rescaled limit, 
and is smooth, in contrast to the variety at $c=0$. 
Although the results of Section~\ref{sec:quadric}
are subsumed by later arguments, its focus on 
explicit computation allows for valuable intuition and visualization.
Properties of our family of perturbations
are given in Section~\ref{sec:perturbation}.  
Section~\ref{sec:relative} uses this approach to complete a relative
homology computation in $\M_c$ for sufficiently small $c$.
It turns out that $\sigma_c$ is not null-homologous in 
$H_{d-1} ((\sing_c)_* , \sing_c (\leq -\ee))$ but is instead 
homologous to an absolute cycle $S_c$ which is homeomorphic 
to a $(d-2)$-sphere and lies in a small neighborhood of $\zz_*$.  
Under the hypotheses of Theorem~\ref{th:lacuna}, the integral 
over $S_c$ is easily seen to converge to zero as $c \downarrow 0$.
The remaining outline of the proof is as follows.

\begin{enumerate}[1.]
\item Show that $\sigma_c \cong S_c$ in $H_{d-1} ((\sing_c)_* , 
\sing_c (\leq - \ee))$.  This is accomplished in Section~\ref{sec:relative}.
\label{step:1}
\item Pass to a tubular neighborhood to see that $T$ in~\eqref{eq:cauchy} 
may be replaced by the sum of tubular neighborhoods of $S_c$ and
a second chain $\gamma$, not depending on $c$, whose maximum height 
is at most $h(\zz_*) - \ee$.  
This is accomplished in Section~\ref{sec:proof}.
\label{step:2}
\item Dimension analysis shows that the integral over the tubular 
neighborhood of $S_c$ goes to zero as $c \downarrow 0$.  This
is accomplished in Section~\ref{sec:proof}, proving 
Theorem~\ref{th:lacuna}.
\label{step:3}
\item Further Morse theoretic analysis shows that the contour $\gamma$
is an integer times the sum of two standard saddle-point contours. 
In Section~\ref{sec:GRZ} we show that for our motivating
example this integer is in fact~3,
proving Theorem~\ref{th:zeta}
\label{step:4}
\end{enumerate}

\section{Preliminaries: tubes, intersection class, residue form}
\label{sec:preliminaries}

We recall some topological facts from various sources, most
of which are summarized for application to ACSV 
in~\cite[pages~334~--~338]{PW-book}.
Let $K$ be any compact subset of $\sing_*$ on which 
the gradient of the square-free part of $Q$ 
(the product of its distinct irreducible factors)
does not vanish.  The well known Tubular Neighborhood Theorem 
(for example,~\cite[Theorem~11.1]{milnor-stasheff}) states the 
following.
\begin{pr}[Tubular Neighborhood Theorem] \label{pr:tube}
The normal bundle over $K$ is trivial and there is a global
product structure of a tubular neighborhood of $\sing_*$ in $\C_*^d$.
$\Cox$
\end{pr}
This implies the existence of operators $\disk$ and $\leray$,
respectively the product with a small disk and with its boundary,
mapping $k$-chains in $\sing_*$ respectively to $(k+2)$-chains
in $\M$ and $(k+1)$-chains in $\M$, well defined up
to a natural homeomorphism as long as the radius of the disk
is sufficiently small.  We refer to $\leray \chain$ as the 
{\bf tube around $\gamma$} and $\disk \chain$ as the 
{\bf tubular neighborhood} of $\gamma$.  Elementary rules 
for boundaries of products imply
\begin{equation} \label{eq:bdry disk}
\begin{array}{rcl}
\partial (\leray \gamma) & = & \leray (\partial \gamma) \, ; \\
\partial (\disk \gamma) & = & \leray \gamma \cup \disk (\partial \gamma) \, .
\end{array}
\end{equation}
Because $\leray$ commutes with $\partial$, cycles map to cycles,
boundaries map to boundaries, and the map $\leray$ on
the singular chain complex of $\sing_*$ induces a map on homology 
$\leray : H_* (\sing_*) \to H_* (\C_*^d \setminus \sing)$.  This 
allows one to construct the {\bf intersection class}, as 
in~\cite[Proposition~2.9]{BMP-morse}.

\begin{defn}[intersection class]\label{def:IC}
Suppose $Q$ vanishes on a smooth variety $\sing$ and let $\TT$ and $\TT'$
be two $d$-cycles in $\M$ that are homologous in $\C_*^d$.  Then there exists
a unique class $\I = \I (\TT , \TT') \in H_{d-1} (\sing_*)$ such that
$$[\TT] - [\TT'] = \leray \I \qquad \mbox{\rm in } H_d (\M) \, .$$
The class $\I$ can be represented by the manifold $\bH \cap \sing$
for any manifold $\bH$ with boundary $\TT-\TT'$ in $\C_*^d$ that
intersects $\sing$ transversely, with appropriate orientation (or, 
alternatively, by the image of the fundamental class of 
$\bH\cap\sing$ under the natural embedding).
\end{defn}

We remark that if $\sing$ is not smooth but its singularities 
(where $Q$ and the gradient of its square-free 
part vanish) have real dimension less than 
$d-2$ then $\bH$ generically avoids the singularities of $\sing$, so
$\I(\TT , \TT')$ is well defined. Although the singular set does not 
always satisfy this dimensional condition, it does so in our applications,
where the singular set is zero dimensional.

For our purposes, the natural cycles to consider are the tori $\TT(\xx)$ 
for $\xx$ in the complement of the amoeba of $Q$. In this case, there is an 
especially convenient choice of cobordism between $\TT(\xx)$ and $\TT(\xx')$, 
namely the $\Ll$-preimage of the straight segment connecting $\xx$ and 
$\xx'$ (or its small perturbation). We will be referring to this cobordism 
as the {\em standard} one.

What are the good choices of $\xx'$? We would like to make the integrand 
$F(\zz)\zz^{-\rr}d\zz/\zz$ exponentially small in $|\rr|$ when 
$\Ll(\zz) =\xx'$, which happens if we can take $-\rr\cdot\xx'$ 
to have arbitrarily small
modulus. When $Q$ is a 
Laurent polynomial the feasibility of this follows from known facts
about cones of hyperbolicity, as we now demonstrate.

First, recall that the Newton polytope of $Q$ is the convex hull of the 
exponents $\mm$ of the monomials of $Q$,
\begin{equation}\label{eq:newton}
N(Q)=\mathtt{conv}(\{\mm: q_\mm\neq 0, Q(\zz)=\sum_\mm q_\mm \zz^\mm\})\subset\Real^d.
\end{equation}
The Newton polytope has vertices in the integer lattice, and the 
convex open components of the
amoeba complement $\amoeba(Q)^c$ map injectively into the integer
points in $N(Q)$ (see \cite{FPTs00}). Moreover, any vertex of $N(Q)$
has a preimage under this mapping, which is an unbounded 
component of $\amoeba(Q)^c$. The recession cone of the component 
$B_\mm$ corresponding to a vertex $\mm$ is the interior of the 
normal cone to $N(Q)$ at $\mm$ (i.e., the collection of vectors 
$\dd$ such that $\max_{\rr\in N(Q)} (\dd,\rr)$ is uniquely 
attained at $\mm$). Notice that this normal cone is dual to the 
cone $N(Q)_\mm$ spanned by $N(Q)-\mm$.

Now, let $B$ be the component of $\amoeba(Q)^c$ corresponding
to the Laurent expansion of $F$ under consideration, and let $\mm$ be the corresponding
integer vector in $N(Q)$. The vectors $\vv$ with $\xx_*+\vv \in B$ form an open 
cone contained in $N(Q)_\mm$. Pick a generic $\dd$ in the 
recession cone of $B$; then when $t>0$ is large enough $\xx_*-t\dd$ 
is contained in an unbounded component $B'$ of the complement 
to amoeba (this follows from the fact that the union of the 
recession cones of the unbounded components of $\amoeba(Q)^c$
are the complement to the set of functionals attaining their 
maxima on $N(Q)$ at multiple points, a positive codimension fan in $\Real^d$).
Hence, choosing $\xx'$ in this component $B'$ allows one to 
deform $\TT'=\TT(\xx')$ while avoiding $\sing$ so that $h_\rr$ becomes 
arbitrarily close to $-\infty$.

\begin{defn}\label{def:descending comp}
We call a component $B'$ whose recession 
cone contains a vector $-\dd$ with $\dd$ in the recession 
cone of $B$ {\em descending} with respect to the component 
$B$. Components $B'$ with this property are in general 
not unique, but any choice of $B'$ works for our argument.
\end{defn}

The following result is well known; see, 
e.g.~\cite[Proposition~2.14]{BMP-morse}.

\begin{defn}[residue form] \label{def:residue}
There is a homomorphism $\Res : H^d (\M) \to H^{d-1} (\sing_*)$ 
in deRham cohomologies such
that for any class $\gamma \in H_d (\sing)$,
\begin{equation} \label{eq:residue}
\int_{\leray \gamma} \omega = \int_{\gamma} \Res (\omega) \, .
\end{equation}
In general, $\Res(\omega)$ can be derived 
locally from a form representing $\omega$ 
(we also use the notation $\Res$ for the 
corresponding operator on differential forms). 
When $F = P/Q$ is rational 
with $Q$ squarefree, $\Res$ commutes with multiplication by 
any locally holomorphic function and satisfies
$$Q \wedge \Res (F \, d\zz) = P \, d\zz \, .$$
More generally, if $F = P / Q^k$, then (see, e.g. \cite{Pham})
the residue can be expressed in coordinates as
\begin{equation} \label{eq:residue k}
\Res \left[ \zz^{-\rr} F(\zz) \frac{d\zz}{\zz} \right ]
   := \frac{1}{(k-1)!}\frac{d^{k-1}}{dc^{k-1}} \left [
   \frac{P \zz^{-\rr}}{\zz} \right ] \, d\sigma\, ,
\end{equation}
where $\sigma$ is the natural area form on $\sing$ (characterized by $dQ\wedge \sigma=d\zz$), and the 
partial derivatives with respect to $c$ are taken in the coordinates 
where $c$ is one of the variables.
\end{defn}

Putting this together with the definition and construction of the
intersection class and Cauchy's integral formula yields the following 
representation of the coefficients $a_\rr$.

\begin{pr} \label{pr:intersection}
Suppose $F = G/Q^k = \sum_{\rr \in E} a_\rr \zz^\rr$ with $G$ holomorphic
and $Q$ a polynomial, the series converging when $\log |\zz|$ is 
in the component $B$ of $\amoeba (Q)^c$.  Let 
$\xx \in B$ and $\TT (\xx) := \Ll^{-1} (\xx)$ be the torus with 
log-polyradius $\xx$.  Let $\xx'$ be any other point in $\amoeba(Q)^c$.  
Then 
\begin{equation}\label{eq:two terms}
a_\rr \qquad = \qquad 
\frac{1}{(2 \pi i)^d} \int_{\I (\TT (\xx) , \TT(\xx'))} 
   \Res \left [ \zz^{-\rr} F(\zz) \frac{d\zz}{\zz} \right ]
   \; \; + \;\;  
   \frac{1}{(2 \pi i)^d} \int_{\TT (\xx')} \zz^{-\rr} F(\zz) \frac{d\zz}{\zz} 
   \,.
 \end{equation}
Moreover, if $\xx'$ is a descending component $B'$ with 
respect to $B$, and $G$ is a polynomial, then 
for all but finitely many $\rr \in E$,
$$a_\rr = \frac{1}{(2 \pi i)^d} \int_{\I (\TT(\xx) , \TT(\xx'))} \zz^{-\rr}
   \Res \left ( F(\zz) \, \frac{d\zz}{\zz} \right ) \, .$$
\end{pr}

\noindent{\sc Proof:} The first identity is Cauchy's integral formula, the
definition of the intersection class, and~\eqref{eq:residue}. 
The second identity follows from the fact 
that $\sup_{\TT(\xx')}|G/Q^k|$ and the volume of $\TT(\xx')$ grow
at most polynomially in $|\xx'|$ on the torus over $\xx'$. 
For $\rr\in N(Q)_\mm$  large enough in size, the degree 
of the decay of $|\zz^{-\rr}|$ overtakes that polynomial growth, 
so that the last term in \eqref{eq:two terms} can be 
made arbitrarily small. As it is independent of $\xx'$ 
as long as $\xx'$ varies in the same component $B'$, it vanishes identically.
$\Cox$

\section{The limiting quadric} \label{sec:quadric}

In this section we focus on the properties of the particular 
smooth quadratic function
\begin{equation} \label{eq:q}
q(\zz) := -1 + z_1^2 - \sum_{j=2}^d z_j^2  \, .
\end{equation}
Note that unlike the local behaviour of a function $Q$ near a
real hyperbolic quadratic singularity,
the quadric $q$ has constant term $-1$.
The zero set $\singq$ of $q$ can be viewed as 
the solution set $\sing_c$ to the equation $Q(\zz) = c$ 
near the quadratic singularity of $Q$,
after the variables are scaled by $c^{1/2}$.
Our first statement deals with the gradient-like flow on 
$\singq$ with respect to the function $h:=x_0$.

\begin{lem}\label{lem:grad_quadric}
The function $h$ has two critical points $\zz_\pm=(\pm1,0,\ldots,0)$ 
on $\singq$, both of index $d-1$. The stable 
manifold for $\zz_+$ is the unit sphere 
$$ \sph:=\left\{\xx+i\yy : x_1^2+\sum_{k=2}^d y_k^2=0, \; y_1=x_2=\cdots=x_d=0\right\} $$
and its unstable manifold is the upper lobe of the 2 sheeted real hyperboloid 
$$ \hyper_+ := \singq \cap \Real^d = 
\left\{\xx+i\yy :  x_1^2 - \sum_{k=2}^d x_k^2 = 0, \; 
  x_1>0,\; \yy=\zero \right\} \, .  $$
The stable manifold for $\zz_-$ is the lower lobe $\hyper_-$ 
of this hyperboloid, while the unstable manifold of $\zz_-$ 
is still the sphere $\sph$.
\end{lem}

\begin{proof}
The critical points can be found by a direct computation.  
Their indices are necessarily $d-1$, as $h$ is the real part 
of a holomorphic function on a complex manifold~\cite{GH}. 
Similarly, direct computation shows that the tangent spaces to 
$\sph, \hyper_\pm$ are the stable/unstable eigenspaces 
for the Hessian matrices of $h$ restricted to $\singq$
at the critical points. 
Lastly, as the gradient vector field is invariant 
with respect to symmetries $\yy\mapsto -\yy$ and 
$(x_1,x_2,\ldots,x_d)\mapsto(x_1,-x_2,\ldots,-x_d)$, 
leaving $\hyper_\pm$ and $\sph$ invariant, they are 
the invariant manifolds for the gradient flow.
\qed
\end{proof} 

Let $\homot : \R^d \times [-1,1] \to \C^d$ be the homotopy defined by
$\homot (\yy , t) := t \e_1 + i \yy$ (we use a new symbol because
$\bH$ is in principle only a cobordism).  Let $\hq$ denote the height
function $\hq(\zz) = - \Re \{ z_1 \}$.  

\begin{thm} \label{th:hyp-sphere}
The intersection cycle of the homotopy $\homot$ with the variety 
$\singq$ is the union of a hyperboloid $\hyper$ and a $(d-1)$-sphere
$\sph$ that intersect in a $(d-2)$-sphere $\sphh$.  These are given
by equations~\eqref{eq:hyper}~--~\eqref{eq:sphh}.  The orientation
of the intersection cycle is continuous on each of the four smooth 
pieces, namely the upper and lower half of $\hyper \setminus \sphh$
and the northern and southern hemispheres of $\sph \setminus \sphh$,
but change signs when crossing $\sphh$.
\end{thm}

\noindent{\sc Proof:}
Writing $\zz_j := x_j + i y_j$, the equations for $\zz$ such that
$\zz$ is in the range of $\homot$ and $q (\zz)= 0$ become
\begin{eqnarray} 
|x_1| & \leq & 1 \label{eq:x1} \\
x_j & = & 0 \qquad (2 \leq j \leq d) \label{eq:xj} \\
x_1^2 - y_1^2 & = & 1 - \sum_{j=2}^d y_j^2 \label{eq:Re} \\
x_1 y_1 & = & 0 \label{eq:Im} \, .
\end{eqnarray}
The solutions to~\eqref{eq:x1}~--~\eqref{eq:Im} form
the union of two sets, one obtained by 
solving~\eqref{eq:x1}~--~\eqref{eq:Re} when $x_1 = 0$ and the 
other by solving~\eqref{eq:x1}~--~\eqref{eq:Re} when $y_1 = 0$;
these intersect along the solutions to~\eqref{eq:x1}~--~\eqref{eq:Re} 
when $x_1 = y_1 = 0$.  The first of these is the one-sheeted hyperboloid
$\hyper \subseteq i \R^d$ given by
\begin{equation} \label{eq:hyper}
- y_1^2 = 1 - \sum_{j=2}^d y_j^2 \, .
\end{equation}
The second is the sphere $\sph \subseteq \R \times i (\R^{d-1})$ given by 
\begin{equation} \label{eq:sph}
x_1^2 + \sum_{j=2}^d y_j^2  = 1 \, .
\end{equation}
These intersect at the equator of the sphere $\sph$, which is the neck
of the hyperboloid $\hyper$.  The intersection set is the sphere $\sphh$
in $\{ 0 \} \times i \R^{d-1}$ given by
\begin{equation} \label{eq:sphh}
\sum_{j=2}^d y_j^2  = 1 \, .
\end{equation}

The intersection class is given by the intersection of $\singq$ with any
homotopy intersecting it transversely.  While $\homot$ does not intersect
$\singq$ transversely, it is the limit of the intersections of $\singq$
with arbitrarily small perturbations of $\homot$ that do intersect
$\singq$ transversely.  Let $\gamma_n$ be a sequence of such transverse
intersection cycles converging to $\gamma := \hyper \cup \sph$.  
Because $\singq$ is smooth, the global product structure on a 
neighborhood of $\singq$ from the Thom lemma implies that as $d$-chains,
\begin{eqnarray*}
\homot (\cdot , -1) - \homot (\cdot , 1)  = 
   \gamma_n \times S_1 \to  \gamma \times S_1,
\end{eqnarray*}
and hence that $\gamma$ represents the intersection class.

Finally, we determine the orientation via a different perturbation
argument.  Choose a point $p \in \sphh$, say for specificity
$p = (0 , \ldots , 0 , i)$.  The tangent space $T_p (\sphh)$
is the span of the vectors $i \e_k$ for $2 \leq k \leq d-1$.
The tangent space $T_p (\sph)$ is obtained by adding the
basis vector $\e_1$, while the span of the tangent space 
$T_p (\hyper)$ is obtained by adding instead the basis vector $i \e_1$.
We see that near $\sphh$, $\gamma$ has a product structure
$\sphh \times \WW$, where $\WW$ is diffeomorphic to two crossing lines, 
with tangent cone $xy = 0$ in the plane $\langle \e_1 , i \e_1 \rangle$,
as in the black lines in Figure~\ref{fig:W}.
\begin{figure}
\centering
\raisebox{-0.6in}{\parbox{2.8in}{\includegraphics[width=2.8in]{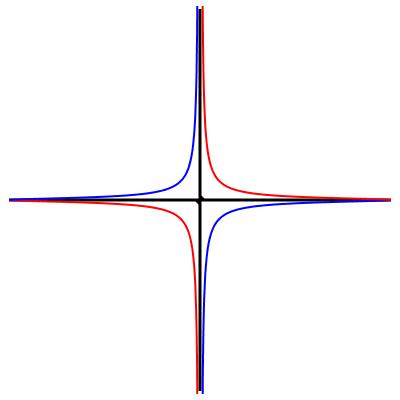}}}
\hspace{0.3in}
\parbox{2.8in}{\includegraphics[width=2.4in]{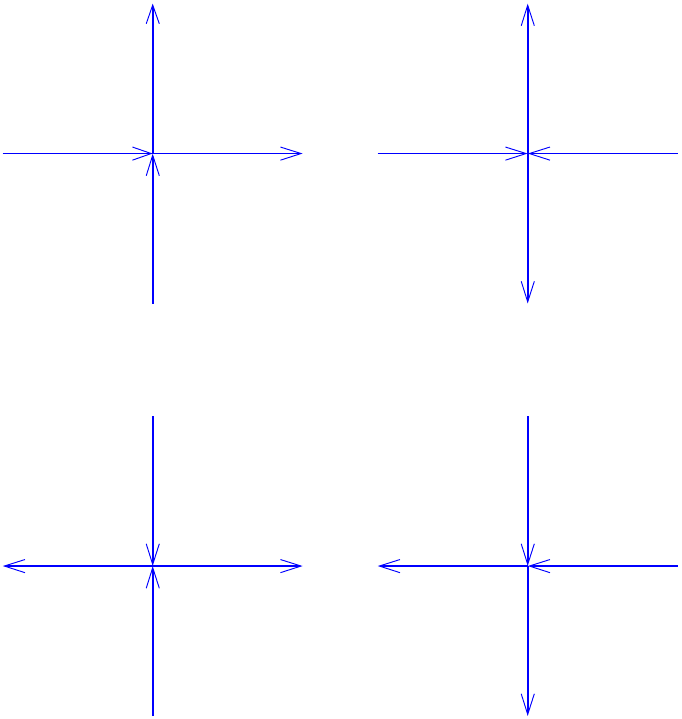}}
\caption{(left) The black line shows $\WW$; the blue line shows the projections 
to the $x_1$-$y_1$ plane of $\WW_\ee$ when $y_d > 0$ and $\ee$ 
is small and positive; the red line shows the projections when
$\ee$ is small and negative; (right) orientations of $\WW$ consistent 
with the blue hyperbola.}
\label{fig:W}
\end{figure}

Now perturb the homotopy as follows.  Let $u: [-1,1] \to \R$ be a 
smooth function that is equal to~1 on $[-1/4 , 1/4]$ and vanishes
outside of $(-1/2 , 1/2)$. 
Define $\homot_\ee (\yy , t) := t \e_1 + \ee u(t) \e_d + i \yy$ where 
$\ee$ is a real number whose magnitude will be chosen sufficiently small 
and whose sign could be either positive or negative.  Because $\sph$
and $\hyper$ intersect only on the subset of $\homot$ where $t = 0$,
their Hausdorff distance on the set $t \notin (-1/4,1/4)$ is positive;
it follows that for sufficiently small $|\ee|$, the intersection of
$\homot_\ee$ with $\singq$ is in the subset of $\homot_\ee$ where
$-1/4 \leq t \leq 1/4$.  There $u = 1$ and the equations for the
intersection $\gamma_\ee$ are modified from~\eqref{eq:x1}~--~\eqref{eq:Im}
as follows:
\begin{eqnarray*}
|x_1| \leq 1 && \eqref{eq:x1} \\
x_j = 0 \;\; (2 \leq j \leq d-1) , \;\; \; x_d = \ee && \eqref{eq:xj}' \\
x_1^2 - y_1^2 = 1 - \ee^2 - \sum_{j=2}^d y_j^2 && \eqref{eq:Re}' \\
x_1 y_1 = - \ee y_d \, . && \eqref{eq:Im}'
\end{eqnarray*}

Although $\homot_\ee$ and $\singq$ still do not intersect transversely,
the intersection set $\gamma_\ee := \homot_\ee \cap \singq$ is now a 
manifold.  We now fix $y_2, \ldots , y_{d-1}$ at a value $\yy$ inside the
unit ball, setting $x_2 = \ee$ and solving~\eqref{eq:Re}$'$ 
and~\eqref{eq:Im}$'$ for $y_1$ and $y_d$ as a function of $x_1$. 
For $x_1^2 < 1 - |\yy|^2$, as $\ee \downarrow 0$, there are two
components of the solution, with $y_d \to \pm \sqrt{1 - x_1^2 - |\yy|^2}$
respectively.  These correspond to different points on the sphere.
Fixing one, say with $y_d > 0$, locally $\gamma_\ee$ has a product
structure $\sphh \times \WW_\ee$, where $\WW_\ee$ is a hyperbola
in quadrants~II and~IV; see the blue curve in Figure~\ref{fig:W}.  
The (oriented) chains $\WW_\ee$ converge to $\WW$ as $\ee \downarrow 0$, 
therefore the possible orientations for $\WW$ are one of the four
shown on the right of Figure~\ref{fig:W}.
The (oriented) chains $\WW_\ee$ also converge to $\WW$ as $\ee \uparrow 0$, 
narrowing the choices to the second and third choices in Figure~\ref{fig:W},
and proving the desired result.
$\Cox$

\begin{thm} \label{th:null}
Let $\north$ be the chain given by $\sph$ with orientation reversed
in the southern hemisphere; in other words, $\north$ is a sphere,
oriented the same as the northern hemisphere of $\sph$.  When $d$
is even, the chain $\gamma$ is homotopic to $\north$ in $H_{d-1} (\singq)$.
\end{thm}

\noindent{\sc Proof:}  Let $X_1 := \R \times \sphh$ and $\iota_1 :
\sphh \to X_1$ be the embedding $\yy \mapsto (0,\yy)$.  Let
$X_2 = [-\pi/2 , \pi/2] \times \sphh$ and $\iota_2 : \sphh \to X_2$ 
be the embedding $\yy \mapsto (0,\yy)$.  Let $X$ denote the space 
obtained by gluing $X_1$ to $X_2$ modulo the identification of $\iota_1$
and $\iota_2$ (which conveniently identifies identically named points
$(0,\yy)$ in $X_1$ and $X_2$).  If for $j \in\{1, 2\}$ there are homotopies
$T_j : X_j \times [0,1] \to \singq$ making the maps in 
Figure~\ref{fig:ident} commute, then their union modulo the
identification is a homotopy $T : X \times [0,1] \to \singq$.

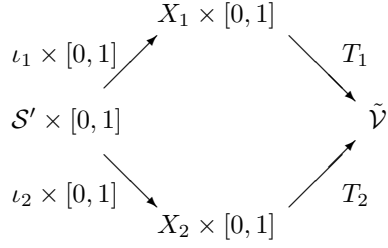
\begin{figure}[!ht] \label{fig:ident}
\hspace{1.5in}
\begin{picture}(100,100)
\put(40,40){$\sphh \times [0,1]$}
\put(75,55){\vector(1,1){20}}
\put(95,80){$X_1 \times [0,1]$}
\put(145,75){\vector(1,-1){25}}
\put(40,65){$\iota_1 \times [0,1]$}
\put(75,30){\vector(1,-1){20}}
\put(95,0){$X_2 \times [0,1]$}
\put(165,65){$T_1$}
\put(40,12){$\iota_2 \times [0,1]$}
\put(145,10){\vector(1,1){25}}
\put(165,12){$T_2$}
\put(175,40){$\singq$}
\end{picture}
\caption{commuting homotopies define a homotopy on the 
identification space $X$}
\end{figure} 

To prove the lemma, it suffices to construct these in such a way
that $T_2$ is a homotopy from $\sph$ to $\north$ and $T_1$ is
a homotopy from $\hyper$ to a null homologous chain.  On $X_1$,
let $\rho$ denote the $\R$ coordinate on $X_1$ and $\sigma$ denote 
the $\sphh$ coordinate.  Let $\zt=(z_2,\dots,z_d)$
and let $\xt$ and $\yt$ denote respectively the real and imaginary 
parts of $\zt$.  Let $t$ denote the $[0,1]$-coordinate of $X_1 \times
[0,1]$.  We may then define the homotopy $T_1$ via the equations
\begin{align*}
\begin{split}
x_0&=\sin(\frac{\pi}2 t)\cosh(\rho), \quad
y_0=\cos(\frac{\pi}2 t)\sinh(\rho), \\[+2mm]
x&=\sin(\frac{\pi}2 t)\sigma\sinh{\rho}, \quad\;\;
y=\cos(\frac{\pi}2 t)\sigma\cosh{\rho},
\end{split}\label{eq:deform}
\end{align*}
and check that $T_1 ((\rho , \sigma) , 0)$ parametrizes $\hyper$ via 
$$y_1 = \sinh (\rho), \qquad \yt = \cosh (\rho) \sigma \, .$$

Next we define the map $\tau : [-\pi/2,\pi/2] \times [0,1] \to 
[-\pi/2 , \pi/2]$ by $\tau (\rho,t) = (1-t) \rho + t (\min(2 \rho, 0) 
- \pi/2)$.  This is a linear homotopy from the identity to the map
$\rho \mapsto \min(2 \rho, 0) - \pi/2$, pictured in Figure~\ref{fig:tau}.
Define $T_2$ by the equations
\begin{equation} \label{eq:deform2}
x_0 = \sin(\tau (\rho,t)), \qquad
y = \cos(\tau (\rho,t)) \sigma \; .
\end{equation}
Again, we verify that $T_2 ((\rho , \sigma) , 0)$ parametrizes the
chain $\sph$ via the parametrization $x_0 = \sin (\rho)$ and $\yt = 
\cos (\rho) \sigma$.  The parametrization is not one to one, mapping
the set $\{ -\pi/2 \} \times \sphh$ to the south pole and $\{ \pi/2 \}
\times \sphh$ to the north pole, however it defines a singular chain
homotopy equivalent to a standard parametrization of~$\sphh$.
\begin{figure}[!ht]
\centering
\includegraphics[width=2in]{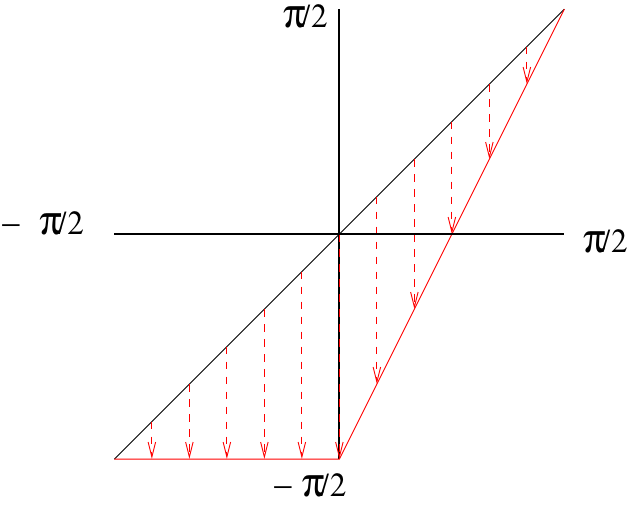}
\caption{the linear homotopy $\tau$}
\label{fig:tau}
\end{figure}

Thirdly, we check that the diagram in Figure~\ref{fig:ident} commutes,
mapping $(\yt , t)$ in both cases to the point $(\sin (t \pi/2) , 
i \cos (t \pi/2) y_2, \ldots , i \cos (t \pi / 2) y_d)$.  Fourthly, we
check that $T_2$ is a homotopy from $\sph$ to $\north$.  This
is clear because the homotopy $T_2$ leaves the (generalized) longitude 
component alone while pushing all the southern latitudes to the south pole
and stretching the northern latitudes to cover all the latitudes.  

Finally, we check that $T_1$ is a homotopy from $\hyper$ to a 
null-homologous chain.  The map $T_1 (\cdot , 1)$ maps the
imaginary hyperboloid $\hyper$ parametrized by 
$(\rho , \sigma)$ into the $\{ x_1 > 0 \}$ branch 
of the real two sheeted hyperboloid $\hyperr$ defined by $x_1^2 = 
1 + \sum_{j=2}^d x_j^2$ and parametrized by cylindrical
coordinates $(r , \sigma')$.  The parametrization is a double
covering, with $(\rho , \sigma)$ and $(-\rho , \sigma)$ 
getting mapped to the same point.  We need to check that
the orientations at $(\rho , \sigma)$ and $(-\rho , -\sigma)$
are opposite.  We may parametrize $\hyper$ by its projection
$\xt$ onto the last $d-1$ coordinates, then, still preserving 
orientation, by polar coordinates $(r , \sigma')$ where $r > 0$
is the magnitude and $\sigma' (\xt) = \xt/r$ when $r > 0$ ($\sigma'$
can be anything when $r=0$).  In these coordinates, the point
$(\rho , \sigma) \in \hyper$ gets mapped to the point
$$ (r, \sigma') = \begin{cases} (\sinh (\rho) , \sigma) & \rho > 0 \\
   (- \sinh (\rho) , -\sigma) & \rho < 0 \end{cases} \, .$$
Recalling that the orientation form on $\hyper$ is given by 
$\sign (\rho) d\rho \wedge d\sigma$, the Jacobian is therefore
given by
\begin{equation}\label{eq:orient}
\frac{D(\sigma',r)}{D(\sigma,\rho)}=
\left\{
\begin{array}{ll}
\frac{d\sigma\wedge\cosh(\rho)d\rho}{d\sigma\wedge d\rho}&:\rho>0\\[+2mm]
\frac{d(-\sigma)\wedge(-\cosh(\rho))d\rho}{-d\sigma\wedge d\rho}&:\rho<0.
\end{array} \right.
\end{equation}

The central symmetry flips the orientation exactly on even-dimensional 
spheres, so that~\eqref{eq:orient} changes signs with the sign of $\rho$ 
exactly when $d-2$ is even.  This implies that for $d$ even, the 
two branches locally covering the sheet $\{x_0^2=|x|^2+1, x_0>0\}$ 
receive opposite signs and the chain $T_1 (\cdot , 1)$ is homologous
to zero.
$\Cox$

In the next section we prove perturbed versions of these results
leading to identification of certain homology and cohomology classes.
To pave the way, we record some further facts about the intersection
of the explicit homotopy $\homot$ with the quadric.

\begin{pr} \label{pr:up down}
There are precisely two critical points for $\hq(\zz) := - \Re \{ z_1 \}$
on $\singq$, namely $\pm e_1$.  At the higher critical point $-\e_1$,
the unstable manifold for the downward gradient flow on $\singq$ is 
the sphere $\sph$, which happens to be a subset of $\homot$, with 
flow lines going longitudinally from the ``north pole'' $-\e_1$ 
to the ``south pole'' $\e_1$.  The stable manifold for the downward 
gradient flow at the north pole is not a subset of $\homot$; it is
the upper sheet $\hyper^+$ of the two-sheeted hyperboloid forming 
the real part of $\sing$, namely the set $\{ \zz \in \R^d : z_1 > 0 
\mbox{ and } z_1^2 = 1 + \sum_{j=2}^d z_j^2 \}$.  
At the south pole $-\e_1$ these are reversed, with the stable manifold
for downward gradient flow equal to $\sph$ and the unstable manifold
being the real surface $\hyper^- := \singq \cap \R^d$;
see Figure~\ref{fig:generators}.
\end{pr}

\begin{figure}[!ht] 
\centering
\includegraphics[width=2.9in]{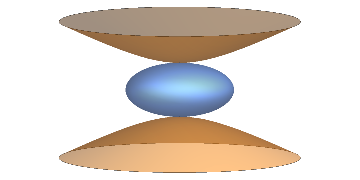}
\caption{Stable and unstable manifolds at the critical points}
\label{fig:generators}
\end{figure}

\noindent{\sc Proof:} Once we check that $\sph, \hyper^+$ and $\hyper^-$
are invariant manifolds for the gradient flow on $\singq$, the proposition 
follows from the dimensions and the fact the ranges of $\hq$ on $\hyper^+$, 
$\sph$ and $\hyper^-$ respectively are $[1,\infty)$, $[-1,1]$ and 
$(-\infty , -1]$.  Invariance of $\hyper^\pm$ follow from the fact 
that the gradient is a real map (the gradient at real points is real)
and therefore the real subspace, of which $\hyper^\pm$ are connected
components, is preserved by gradient flow.  Invariance of $\sph$  
follows from the same argument after reparameterizing via 
$(x_1 , \ldots , x_d) = (s_1 , i \, s_2 , \ldots , i \, s_d)$.
$\Cox$

\section{Perturbation of the variety} \label{sec:perturbation}

Instead of working directly with $Q$, we consider the small perturbations 
$Q_c (\zz) := Q(\zz) - c$. As above, let $\sing_c$ denote the zero set of $Q_c$, 
let $\omega_c=(P/Q_c^k)d\zz/\zz$ denote the corresponding Cauchy $d$-form, and let
$\M^c=\Comp_*^d-\sing_c$ denote the points where $\omega_c$ is analytic.
Below we collect several results on the behavior of this deformation.

\begin{pr}[stable behavior] \label{pr:stable}
Under the setup of the previous paragraph,
\begin{enumerate}[(i)]
\item For sufficiently small $|c|> 0$ the variety $\sing_c$ is smooth. 
\label{item:1}
\item For any index $\rr$, the coefficient of the power series expansion 
for $F_c=P/Q_c^k$ given by \eqref{eq:cauchy},
$$ a_\rr(c):=\left ( \frac{1}{2 \pi i} \right )^d \int_T \zz^{-\rr} 
   \frac{P(\zz)}{Q_c^k} \frac{d\zz}{\zz},$$
is holomorphic in the disk $|c|<|Q(0)|$.  In particular, any given 
coefficient is continuous at $c=0$ as a function of $c$.
\label{item:2}
\end{enumerate}
\end{pr}

\noindent{\sc Proof:} 
The first statement is follows from the Bertini-Sard theorem (the values
of $c$ that make $\sing_c$ singular is a finite algebraic set).  
The second follows 
from the fact that each term in the (converging, under our assumptions) 
expansion of
$$ \frac{P}{(Q-c)^k}=\sum_{l\geq0} {{-k}\choose{l}} \frac{P}{Q^{k+l}} c^l $$
is holomorphic and thus integrable over any torus in the domain 
of holomorphy of $F$, and the modulus of each term is bounded.
$\Cox$

We will need to understand the local behavior of $h_\rr$ on the smooth 
varieties $\sing_c$ near $\zz_*$.  The following proposition shows that
the perturbed varieties have the same geometry as the limiting quadric
described in Section~\ref{sec:quadric}.
\begin{pr}[local behavior] \label{pr:loc}
Assume that $\rhat$ strictly supports the tangent cone $T_{\xx_*} (B)$.
Then
\begin{enumerate}[(i)]
\item There is a $\delta > 0$ such that for sufficiently small  
$|c|\neq 0$, there are precisely two critical points of $h_{\rhat}$ 
on the variety $\sing_c$ in the 
ball $B_\delta (\zz_*)$.  These points tend to $\zz_*$ as $c \to 0$.
\label{item:3}
\item If $c$ is positive and real, these critical points $\zz_c^\pm$
are real, and can be chosen such that
\begin{equation}
h_{\rhat} (\zz_c^+) > h_{\rhat} (\zz_*) > h_{\rhat} (\zz_c^-) .
\end{equation}
\label{item:4}
\end{enumerate}
\end{pr}

\noindent{\sc Proof:} 
By part~$(i)$ of Proposition~\ref{pr:stable}, 
$\sing_c$ is smooth.
The function $h_{\rhat}$ is the real part of the logarithm of 
the locally holomorphic 
function $\zz^{\rhat}$ near $\zz_*$, hence it has a critical point 
on the smooth complex manifold $\sing_c$ if and only if
$\zz^\rr$ does, i.e., if $d\zz^\rr$ is collinear with $dQ$.  This 
latter condition defines the so-called {\bf log-polar} variety. 
A local computation implies that under our conditions 
this is a smooth curve, intersecting $\sing$ with multiplicity~2
at $\zz_*$ as long as $\rr$ is not tangent to the tangent cone 
$T_{\Ll (\zz_*)} (\log \sing)$.

Indeed, one can find a real affine-linear coordinate change such that 
in the new coordinates, centered at $\zz_*$, 
$$Q = z_1^2 - \sum_{k \geq 2} z_k^2 + O(|\zz|^3) \qquad\text{and}\qquad
   h = z_1 + \sum_{k \geq 2} a_k z_k + O(|\zz|^2) \, ,  $$
where our conditions on $\rr$ imply $\sum_k a_k^2 < 1$.
In these coordinates, the log-polar variety is given by the 
equation $z_k = -a_k z_0 + O(|\zz|^2)$. Thus,
the log-polar curve intersects $\sing_c$ transversely for 
$|c| \neq 0$ small, and consists of $2$ geometrically distinct points.
A similar computation implies the second statement for real $c$.
$\Cox$

The main work in proving Theorem~\ref{th:lacuna} will be to prove 
the following result. 

\begin{thm} \label{th:actual} 
Assume the hypotheses of Theorem~\ref{th:lacuna}. For
$\ee>0$ and $c_*>0$ small enough, and for any
$\rhat \in K \subseteq \nbdhat$, there is a cycle $\Gamma(\rhat)$
such that for $|c|<c_*$,
\begin{enumerate}[(i)]
\item The cycle $\Gamma(\rhat)$ lies in the set $\M_c(-\ee)$; in other words,
the cycle $\Gamma(\rhat)$ lies below the height level $h_{\rhat}(\zz_*) - \ee$ 
for all $\rhat \in K$, and it avoids $\sing_c$ for all $c$ such that $|c|<c_*$.
\item There is a chain 
$\Gamma_c \subseteq \M^c$ such that $[\TT] \simeq [\Gamma_c] + [\Gamma(\rhat)]$ 
in $H_d (\M_c)$.
\item The cycle $\Gamma(\rhat)$ can be chosen 
to be $[\leray\gamma(\rhat)]+[\TT(\yy)]$, where $\gamma(\rhat)$ 
is a $(d-1)$ cycle in $\sing(-\ee)$, and $\yy$ is in a descending
component $B'$ of the complement of $\amoeba (Q)$ with respect to $B$
(see Definition \ref{def:descending comp}).
\item For fixed $\rr$ as $c \to 0$, 
\begin{eqnarray}
\int_{\Gamma} \zz^{-\rr}\omega_c & \to & \int_\Gamma  \zz^{-\rr}\omega \hspace{1.5in} 
   \label{eq:omega_c} \\
\int_{\Gamma_c}  \zz^{-\rr}\omega_c & \to & 0 \label{eq:gamma_c} , \mathrm{\ if\ } d>2k.
\end{eqnarray}
\end{enumerate}
\end{thm}

Theorem~\ref{th:actual} is proven in Section~\ref{sec:proof}.

\noindent{\sc Proof of Theorem~\ref{th:lacuna}:} 
The first statement of Theorem~\ref{th:lacuna} 
follows immediately from Theorem~\ref{th:actual},
as
\begin{equation}
a_\rr = \lim_{c \downarrow 0} a_{\rr,c} =\lim_{c \downarrow 0} \int_{\TT} \zz^{-\rr} \omega_c =
\lim_{c \downarrow 0} \left [ \int_{\Gamma(\rhat)} \zz^{-\rr} \omega_c 
   + \int_{\Gamma_c} \zz^{-\rr} \omega_c.\right ] = \int_{\Gamma(\rhat)} \zz^{-\rr} \omega 
\end{equation}
by~\eqref{eq:omega_c},~\eqref{eq:gamma_c} and~\eqref{eq:exact}. The 
uniform bound in $\rhat$ follows from compactness of $K$. Indeed, 
any cycle $\Gamma(\rhat)$ satisfies the conclusions for all $\rhat'$ 
in small enough open vicinity of $\rhat$; choosing a finite cover of 
$K$ by such open vicinities, we prove the claim.
To obtain the second statement of Theorem~\ref{th:lacuna}, 
we use Proposition \ref{pr:intersection} to see 
that
\begin{equation}
\int_{\TT(\xx')} \zz^{-\rr} \omega_c 
\end{equation}
vanishes for all but finitely many $\rr$. Together with~$(iii)$
of Theorem~\ref{th:actual}, this implies the conclusion of 
Theorem~\ref{th:lacuna} for polynomial numerators.
$\Cox$

\section{Local homology near quadratic point} \label{sec:relative}

Recall our sign choice for $Q$, which implies that $Q$ is positive 
on the real part of the domain of holomorphy for the Laurent
expansion under consideration.  
We are interested in the local topology of the intersections of the 
singular set $\sing_c$ with the height function $h=|\zz^\rr|$. 
We start with a result proved in~\cite[Lemma~1.3]{AVG2}, though it
dates back at least to~\cite{milnor-hypersurfaces}.

\begin{pr} \label{pr:total space}
There exist $\delta, \delta' > 0$ such that if $B = B(\zz_*,\delta)$ 
denotes the ball of radius $\delta$ about $\zz_*$ then 
 $\sing_c \cap B$ is diffeomorphic to the total space of 
the tangent bundle to the $(d-1)$-dimensional sphere
for all $c\in\C$ with $0<|c| < \delta'$.  In particular, 
the (absolute) homology groups of $\sing_c \cap B$ are trivial in dimensions 
not equal to $d-1$, and $H_{d-1}(\sing_c\cap B) \cong \Int$.
$\Cox$
\end{pr}

Let $h_*:=h(\zz_*)$. What we require for our results is 
a description of the relative homology group 
$H_{d-1} ( (\sing_c)_* \cap B , \sing_c \cap B (h \leq h_*-\ee)),$ 
together with explicit generators.  To compute these we start with the
homogeneous situation and then perturb.  Denote by $q$ the 
quadratic part of $Q$ at $\zz_*$.  This is a real quadratic form, 
invariant with respect to conjugation, with signature $(1,d-1)$ 
on the real part of the tangent space at $\zz_*$.  We denote the
two convex cones where $q \geq 0$ as $C_\pm$, and extend
Definition~\ref{def:tangent_cone} by considering supporting
vectors to $C_+$ as well as $C_-$.

Consider the following three surfaces in $\C^d$ of respective 
co-dimensions 1, 1 and 2: $(i)$ the boundary $S$ of the unit ball; 
(ii) the hyperplane $H := \{ \xx + i \yy : \xx \cdot \rhat = 0\}$ 
orthogonal to the real vector $\rhat$; and $(iii)$ the complex hypersurface 
$v := \{ q=0 \}$ defined by the quadric.  The transverse intersection 
of $S$ and $H$ is the equator of~$S$.  

\begin{lem}\label{lem:transversal}
If $\rr$ is supporting then $v$ intersects $S \cap H$ transversely.
\end{lem}
\begin{proof}
By the hypothesis that $\rr$ is supporting, one can choose $h$ as 
one local coordinate, changing the rest of the coordinates so that 
the quadratic form $q$ preserves its Lorentzian form.  In these new 
coordinates it remains to prove that the functions
$$ 
x_1=0, 
\quad
x_1^2 + \sum_2^d y_k^2 - y_1^2 + \sum_2^d x_k^2 = 0, 
\quad 
 x_1 y_1 - \sum_2^d x_k y_k=0 $$
have independent differentials at their common zeros outside of 
the origin.  This can be checked directly. \qed
\end{proof}

\begin{cor}\label{cor:slab}
For $\rho > 0$ small enough there are positive numbers $\epsilon_*$ 
and $c_*$ such that the manifolds 
$\{\zz:\|\zz-\zz_*\| = r\},$  $\{\zz:h(\zz) = h_*(\zz) + \epsilon\}$ 
and $\{\zz:Q(\zz)=c\}$ intersect 
transversely, provided that $\rho/2 < r < \rho$ while $\epsilon < \epsilon_*$
and $|c| <  c_*$.
\end{cor}

\begin{proof}
For a given $\rho>0$, introduce new coordinates in which $\zz_*$ 
is the origin and the $\rho$-ball around $\zz_*$ becomes the unit ball 
in $\Comp^d$, while rescaling $Q$ by $\rho^{-2}$ and $h$ by $\rho^{-1}$. 
The resulting functions become small perturbations (decreasing with $\rho$) of 
the quadratic and linear functions in Lemma \ref{lem:transversal}, 
and their zero sets become small deformations $Q^\rho$ and $H^\rho$ 
of the corresponding varieties.

In particular, the determinants whose nonvanishing witnesses the transversality 
of the varieties of $Q^\rho, H^\rho$ and $S$ are small deformations of the 
determinants witnessing the transversality in Lemma~\ref{lem:transversal}, 
and therefore are non-vanishing on some open neighborhood $U$ of the set 
of solutions to 
$H^\rho=Q^\rho=0$ intersected with the spherical shell where the 
distance to the origin is between, say, $1$ and $1/2$ for small 
enough $\rho$. 

For small enough $\epsilon_*,c_*>0$ the sets $\{|Q^\rho|\leq c_*\}
\cap \{|h^\rho|\leq \epsilon_*\} \cap B_1$ are contained in $U$.
Therefore the varieties $\{Q^\rho=c\}, \{H^\rho=\epsilon\}$ and 
$\{|\zz| = r \}$ are transverse when $|c|\leq c_*, |\epsilon| \leq \epsilon_*$
and $1/2 \leq r \leq 1$.
\qed\end{proof}

We will need one more result on the local geometry of $\sing$ and $\{h=\mathrm{const}\}$.

\begin{lem}\label{lem:bottom}
For $\ee\neq 0$, the intersection of the real hyperplane $x_1=-\ee$ with the quadric
\[
z_1^2-z_2^2-\ldots-z_d^2=c
\]
is homotopy equivalent to a $(d-2)$-dimensional sphere for $|c|$ small enough.
\end{lem}
\begin{proof}
Rescaling, we can assume that $\ee=-1$. Parameterizing $(x_2,\ldots,x_d) = s \xi$
and $(y_2,\ldots,y_d) = t\eta$ where $s,t\geq 0$ and $\xi, \eta$ 
are unit vectors in $\Real^{d-1}$, we obtain the equations
\begin{equation}\label{eq:bottom}
x_1=1, \qquad 
1-y_1^2+t^2|\eta|^2-s^2|\xi|^2=c, \qquad 
y_1= st(\eta\cdot\xi).
\end{equation}
Suppose $c=0$. Then the manifold in question is given by
\[
1+t^2=s^2t^2|\xi\cdot\eta|^2+s^2.
\]
Since $s^2t^2|\xi\cdot\eta|^2+s^2 \leq s^2(1+t^2)$, one can keep $\xi,\eta$ fixed
and retract $(s,t)$ satisfying this equation to $(1,0)$. 
This retracts the manifold onto the unit 
$(d-2)$-sphere.

For nonzero $c$ it can be verified that the manifolds given 
by~\eqref{eq:bottom} are transverse, and therefore remain transverse 
for small $c$, meaning the intersections are homeomorphic.
\qed
\end{proof}

\begin{cor}\label{cor:bottomtop} Assume that $\rhat$ is supporting.
Then, for $\rho>0$ small enough, there are $\ee,c_*>0$ such that 
\[
\sing_c\cap\{h_{\rhat}=-\ee\}\cap B_\rho(\zz_*)
\]
is homotopy equivalent to $S^{(d-2)}$ for $|c|<c_*$.
\end{cor}
\begin{proof}
We can choose coordinates in which the quadratic part of $Q$ and 
$h_\rhat$ are given by $z_1^2-z_2^2-\ldots-z_d^2$ and $x_1$, respectively. Then, 
repeating the argument in Corollary \ref{cor:slab}, we can view 
a rescaled $Q$ and $h$ as small perturbations of the quadratic and 
linear functions in Lemma \ref{lem:bottom}, and apply transversality.
\qed\end{proof}

We will be referring to the intersection 
$$ \slab := \slab_{\rho , \ee} := B_\rho\cap \{|h-h_*| \leq \epsilon\} \, ,$$
for $\rho, \epsilon$ satisfying the conditions of Corollary \ref{cor:slab}, 
as the $(\rho,\epsilon)$-slab.  We call the intersection of the slab 
with the boundary $\partial B_\rho$ its {\bf vertical} boundary, and
the intersection with $h = h_* - \ee$ its {\bf bottom}.

\begin{figure}[!ht] 
\centering
\includegraphics[width=2.9in]{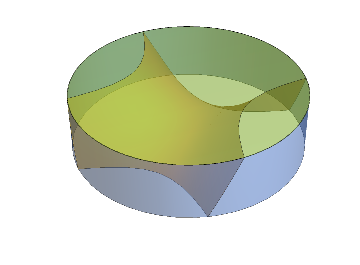}
\caption{A slab.}
\label{fig:slab}
\end{figure}

\begin{cor}\label{cor:vv}
For $\rho, \epsilon_*, c_*$ satisfying the conditions of 
Corollary~\ref{cor:slab}, whenever $|c| < c_*$ there exists 
a vector field $\vv$ on the the intersection $\sing_{c,\slab} 
:= \sing_c\cap \slab_{\rho , \ee_*}$ such that the following hold.
\begin{enumerate}
\item $dh\cdot \vv<0$ everywhere outside of the critical points of $h$ 
on $\sing_{c,\slab}$,
\label{item:slab1}
\item For points on $\sing_{c,\slab}$ within $\rho/3$ from $\zz_*$, 
the vector field is the gradient vector field for $-h$ on $\sing_c$ 
with respect to the standard Hermitian form on $\Comp^d$,
\label{item:slab2}
\item For points at distance between $\rho/2$ and $\rho$ from $\zz_*$, 
the vector field is tangent to the spheres 
$\{|\zz-\zz_*| = \mathtt{const}\}$ and $dh \cdot \vv=-1$,
\label{item:slab3}
\item If $c$ is real, the vector field is invariant under conjugation:
$\vv (\overline{\zz}) = \overline{\vv(\zz)}$.
\label{item:slab4}
\end{enumerate}
\end{cor}

\begin{proof}
Let $\vv^{(\zz_*)}$ denote the gradient vector field for $-h$
on $\sing_c$ as in statement~\ref{item:slab2}.  For any point $\zz$
at distance between $\rho$ and $\rho/2$ from $\zz_*$, 
the transversality conclusions of Corollary~\ref{cor:slab} 
imply that near $\zz$ one can choose coordinates that include the 
four functions $h, |\zz-\zz_*|, \Re \{ Q \}$ and $\Im \{ Q \}$.  
In such coordinates, define $\vv^{(\zz)} := \partial / \partial h$.  
Because $h$, $Q$, distance and the standard Hermitian form are
invariant with respect to complex conjugation, we may choose
the family $\{ \vv^{(\zz)} \}$ to be invariant, in the sense
that $\vv^{(\overline{\zz})} (\overline{\ww})$ is the conjugate
of $\vv^{(\zz)}(\ww)$.  

Use a partition of unity to glue together the vector fields 
$\vv^{(\zz)}$, ensuring that the partition gives weight~1 
to points in a $\rho / 3$ neighborhood of $\zz_*$ and zero weight
outside the $\rho / 2$ neighborhood.  This ensures 
conclusions~\ref{item:slab1},~\ref{item:slab2} and~\ref{item:slab3}.
If the partition is chosen invariant with respect to conjugation, 
the last conclusion will be true as well.
\qed\end{proof}

\begin{pr} \label{pr:local}
Assume again that $\rr$ is supporting at $\zz_*$, where $\zz_*$ 
is a quadratic singularity of $Q$ with signature $(1,d-1)$.  
Fix $\rho$ and $\epsilon$ satisfying conditions of Corollary~\ref{cor:slab} 
and the corresponding $(\rho,\ee)$-slab.  Letting $\bottom$ denote
$\sing_c \cap \slab \cap \{ h - h_* = -\ee \}$,
the relative homology group
$$ H_-:=H_{d-1}(V_c\cap \slab, \bottom ) $$
is free of rank 2 for small enough $|c| \neq 0$.  For small real 
$c > 0$, its generators are given by 
\begin{itemize}
\item
an absolute cycle, the image of the generator of 
$H_{d-1}(V_c \cap B_r)$ under the natural homomorphism into $H_-$, and 
\item
the relative cycle corresponding to the lobe of the real part of $V_c$ 
located in $\{h \leq h_*\}$.
\end{itemize}
\end{pr}

\begin{proof}
The trajectories of the flow along the vector field $\vv (\cdot)$ 
constructed in Corollary~\ref{cor:vv} starting on $\sing_{c,\slab}$ 
either converge to the critical points of $h$ on $\sing_{c,\slab}$, 
or reach $\bottom$.  Indeed, the value of $h$ is strictly decreasing 
outside of the critical points, and cannot leave the slab through 
its side due to conclusion~\ref{item:slab3} of Corollary~\ref{cor:vv}.
All trajectories therefore remain in the slab or reach the bottom.  

The homology of the pair $(\sing_c \cap \slab , \bottom)$ is generated by
classes represented by the unstable manifolds of the Morse function $h$
at critical points on $\sing_c \cap \slab$; this is the fundamental
theorem of stratified Morse theory, for example~\cite[Theorem~B]{GM}.
In our situation, there are exactly two such critical points, 
$\zz_-$ and $\zz_+$, both in the real part of $\sing_c$ and both of index 
$d-1$. This proves the statement about the rank of the group. 

The long exact sequence of the inclusion of the bottom into 
$\sing_c \cap \slab$ gives an exact sequence containing the maps
$$H_{d-1} (\bottom) \to H_{d-1} (\sing_c \cap \slab) \to 
   H_{d-1} (\sing_c \cap \slab , \bottom) \, .$$
Using Corollary~\ref{cor:bottomtop}, the first of these groups vanishes 
because $\sing_c \cap \bottom$ is homotopy equivalent to $S_{d-2}$. 
It follows that the absolute cycle generating $H_{d-1}(V_c \cap B_\rho)$ 
is nonvanishing in $H_{d-1} (\sing_c \cap \slab , \bottom)$ and is
therefore a generator of $H_-$.

For $c>0$, the real part of $\sing_c$ located within the lower half 
of the slab, $\{h<h_*\}$, contains the critical point $\zz_-$ (by
Proposition~\ref{pr:loc}), and 
the vector field $\vv$ is tangent to it (thanks to the reality property 
mentioned above). Hence it coincides with the unstable manifold of $\zz_-$.
\qed\end{proof}

Of course, the same argument applies to the Morse function 
$-h$ on $V_c$, implying that the group
$$H_+ := H_{d-1}(V_c\cap \slab, (V_c\cap\slab) 
   \cap \{h = h_* + \epsilon \}) $$
also has rank~2 and, for positive real $c$, is generated by
the same absolute cycle together with the analogous relative
cycle (the lobe of the real part of $\sing_c$ located in
$\{ h \geq h_* \}$).  
For small positive $c$, the situation we will restrict ourselves to 
from now on, we will denote the generators in $H_-$ as $\sph_-$ and  
$\hyper_-$,  where $\sph_-$ is the absolute class represented by the small sphere 
in $V_c$ and $\hyper_-$ is the relative class represented by the 
corresponding component of the real part of $Q_c$. In the same way 
we define classes $\sph_+$ and $\hyper_+$ generating $H_+$.

A general duality result implies that the relative groups 
$H_-$ and $H_+$ are dual to each other, with the coupling given 
by the intersection index.  Briefly, the reason is that the vector
field in Corollary~\ref{cor:slab} may be used to deform $\slab$
until the boundary of the top flows down to the boundary of the bottom;
this makes the space into a manifold with boundary satisfying the
hypotheses of~\cite[Theorem~3.43]{hatcher}. The conclusion of that
theorem is an isomorphism between a homology group and a cohomology 
group, which, combined with Poincar{\'e} duality, proves the claim.  
In fact, we won't use this argument because we need to compute this 
coupling explicitly, as follows.

\begin{pr} \label{pr:pairing}
The intersection pairing between 
$H_-$ and $H_+$ is given by 
\begin{eqnarray*}
\langle \hyper_+, \hyper_- \rangle & = & 0 \, ; \\
\langle \hyper_+, \sph_- \rangle & = & (-1)^{d(d-1)/2} \, ; \\
\langle \sph_+, \hyper_- \rangle& = & (-1)^{d(d-1)/2} \, ; \\
\langle \sph_+,\sph_- \rangle & = & (-1)^{d(d-1)/2}\chi(S^{d-1})=(-1)^{d(d-1)/2}(1+(-1)^{d-1}) \, .
\end{eqnarray*}
\end{pr}

\begin{unremark}
We pedantically distinguish between $\sph_+$ and $\sph_-$,  
although they are the image of the same absolute class, 
or even chain, in $V_c$. Also, we note that our 
orientations of the spheres and their tangent spaces can be in 
disagreement with the standard orientations induced by the complex 
structure. By changing the orientation of the chain $\sph$, 
one can suppress the annoying sign factor in the second and third 
equalities, but not in the last one.
\end{unremark}

\noindent{\sc Proof:} We can work (after rescaling) in the setup of 
Lemma \ref{lem:grad_quadric}. The cycles representing $\hyper_c^\pm$ are
disjoint, explaining the first line. Each intersect $\sph$ in precisely 
one point.  Denoting $\partial/\partial x_k$ as $\xi_k$ 
and $\partial/\partial y_k$ as $\eta_k$, the tangent spaces 
to $\sph=\sph_\pm$ at $\zz_\pm$ are 
spanned by the vectors
\[
\pm\eta_2,\ldots,\pm\eta_d,
\]
and the tangent spaces to $\hyper_\pm$ at $\zz_\pm$ are spanned by
\[
\pm\xi_2,\ldots,\pm\xi_d,
\]
respectively.

In the standard orientation of the complex hypersurface $\singq$, 
the frame $(\xi_2,\eta_2,\ldots,\xi_d,\eta_d)$ is positive. 
Hence, the intersection index of $\hyper_+$ and $\sph$ is the 
parity of the permutation shuffling 
$$ (\xi_2,\ldots,\xi_d,\eta_2,\ldots,\eta_d) $$
into that standard order, giving the second line. The third line is 
obtained similarly, taking the signs into account.

The last pairing can be observed by noting that the self-intersection 
index of a class represented by a manifold of middle dimension 
in a complex manifold is equal to the Euler characteristics of the 
conormal bundle of the manifold, under the identification of the 
collar neighborhood of the manifold with its conormal bundle. 
This gives $\chi(\sph) = (1+(-1)^{d-1}) (-1)^{d(d-1)/2}$, 
where again the mismatch between the standard orientation 
of the conormal bundle and the ambient complex variety contributes 
the factor $(-1)^{d(d-1)/2}$.
$\Cox$

Importance of the local homology computation lies in the following 
localization result.  Let $\uu_* := \Ll(\zz_*) \in \Real^d$ be a 
point on the boundary of $\amoeba(Q)$ (recall $\Ll$ is the
logarithmic map $\zz \mapsto \log|\zz|$). 
\begin{thm} \label{th:two cp}
Assume that the quadratic critical point $\zz_*$ is the only 
element of $\TT (\uu_*) \cap \sing$, that $\zz_*$ lies on the boundary
of a component of $\amoeba(Q)^c$ 
and that $\rr$ is supporting.  Then for any $\rho > 0$ there exist
$\ee , c_* > 0$ such that for all $c \in (0 , c_*)$ the intersection 
class $\I(\TT) \subseteq \sing_c$ can be represented by a chain 
supported on
$$ B_\rho(\zz_*)\cup \{h\leq h_*-\epsilon\} \, .$$
\end{thm}

\noindent{\sc Proof:}  Choose $\rho$ small enough so that
the conclusions of Corollary~\ref{cor:slab} hold.  As the intersection 
of $\sing$ with the torus $\TT (\uu_*)$ containing $\zz_*$ is a single 
point, standard compactness arguments imply that for sufficiently 
small positive $\delta$ the intersection of $\sing$ with the 
$\Ll$-pre-image of $B(\uu_* , \delta)$ is contained in $B_\rho (\zz_*)$.  
Pick a torus $\TT (\xx)$ where $\xx$ is a point in the intersection 
of $B$ with the component of the complement to the amoeba defining 
our power series expansion.  
Choose $\ee > 0$ such that $\{h \leq h_* - \ee \}$ intersects 
$B_\rho (\zz_*)$.  Let $\yy$ be a point in the component $B'$
defined at the end of Section~\ref{sec:preliminaries}, such
that $h_{\rhat} (\yy) < h_{\rhat} (\zz_*) - \ee$.  Choose any 
smooth path $\{ \alpha (t) : 0 \leq t \leq 1 \}$ from $\xx$ to $\yy$
that passes through $B_\rho (\zz_*)$ and
along which $h_{\rhat}$ decreases.  Then the $\Ll$-preimage 
of that path is a cobordism between $\TT$ and a torus $\TT'$ in 
$\{h \leq h(\zz_*) - \ee \}$.  The transversality conclusion
of Corollary~\ref{cor:slab} means that this cobordism, or a 
small perturbation of it, produces a chain realizing the intersection 
class $\I(\TT)$ and satisfying the desired conclusions.
$\Cox$

We now come to the main result of this section, which
completes Step~\ref{step:1} of the outline at the end of 
Section~\ref{sec:preliminaries}.

\begin{thm} \label{th:rel zero}
For even $d$ the intersection class $\I (\TT)$ is equal
to $[S_c]$ in $H_{d-1} (\sing_c , \sing_c (\leq - \ee))$, up to sign.
\end{thm}

\noindent{\sc Proof:}
Let $e$ denote the class of $\I (\TT)$ in the relative homology group
$H_-$.  Then, by Lemma~\ref{pr:local}, we have $e = a \hyper_- + b S_-$
for some integers $a$ and $b$.  We claim that
\begin{equation} \label{eq:e}
\langle \hyper_+, e\rangle = \pm 1 , \qquad \qquad
   \langle S_+, e \rangle = 0.
\end{equation}

The construction of the chain representing the intersection class
$\I(\TT)$ in Theorem~\ref{th:two cp} implies that it meets the
chain representing $\hyper_+$ at precisely one point $\zz_c'$.
The point $\zz_c'$ is not necessarily the point $\zz_c^+$, but it
is characterized by being the unique point where the homotopy 
intersects the real variety $\sing_{c , \R} \subseteq \sing_c$.

The intersection class is represented by a chain that is smooth 
near $\zz_c'$.  We need to check that its intersection with the 
``upper lobe'' $\hyper_+$ is transverse within $\sing_c$.  
Indeed, one can linearly change coordinates centered at $\zz_c'$ 
so that in the new coordinates $\zz'$ the homotopy segment 
runs along the $x_1'$ axis, and thus the equations defining the 
cobordism are $x'_2 = \ldots = x'_d = 0$.  Then $\sing_c$ is given by
$z_1' = R(z_2' , \ldots , z_d')$ with $dR \vert_{\zz_*^+}=0$. 
By direct computation, the intersection is transversal
and the tangent space to the chain representing the intersection 
class at $\zz_*^+$ is the tangent space to $\sing_c$ at $\zz_*^+$ 
multiplied by $i$.

Because $\hyper_+$ and $e$ intersect transversely at a single point,
the first identity in~\eqref{eq:e} is proved.
For the second identity, we again rely on perturbations of the
cobordism defining the intersection class.  If the path defining the
cobordism avoids $\zz_*$, for $c$ small enough, the chain realizing
$\I(\TT)$ constructed in Theorem~\ref{th:two cp} will completely
avoid the chain representing $\sph$, implying that the intersection
number of $e$ with $\sph_+$ is zero.

To finish, we substitute \eqref{eq:e} into Proposition~\ref{pr:pairing}.
We compute
$$\pm 1 = \langle \hyper_+ , e \rangle = a \cdot 0 + b \cdot \pm 1 \, ,$$
therefore $b = \pm 1$, and
$$0 = \langle \sph_+, e \rangle = \pm a \pm b \chi (S^{d-1}) \, .$$
When $d$ is even the Euler characteristic of the $(d-1)$-dimensional
sphere vanishes together with $a$.
$\Cox$

\section{Proof of the main theorem and Theorem~\protect{\ref{th:actual}}} 
\label{sec:proof}

We are now ready to prove Theorem~\ref{th:actual}, and 
thus obtain our main Theorem~\ref{th:lacuna}.
At each stage it is easiest to prove the result for fixed $\rhat$
and then argue by compactness that the conclusion holds for all
$\rhat \in K$.  We start with a localization result.  Use the
notation $\I_c$ to denote intersection class with respect to 
the perturbed variety $\sing_c$.

\begin{lem}\label{lem:split}
Fix $\rhat \in K$.  Under the hypotheses of Theorem~\ref{th:rel zero}, 
there is an $\ee > 0$ such that the intersection class $\I_c (\TT,\TT')$ 
is 
$$ \I_c = [S_c] + [\gamma_c] \, , $$
where the cycle $\gamma_c (\rhat)$ representing the class $[\gamma_c] 
\in H_{d-1}(\sing_c)$ is supported in $\sing_c(<-\ee)$ with respect to 
$h_{\rhat}$.
\end{lem}

\begin{proof} 
By Theorem~\ref{th:rel zero}, $\I_c - S_c$ is mapped to zero in the
second map of the exact sequence
$$\ldots\to H_{d-1}(\sing_c(<-\ee)) \to H_{d-1}(\sing_c) 
   \to H_{d-1}(\sing_c,\sing_c(<-\ee))\to\ldots$$
Hence $\I_c - [S_c]$ is represented by a class in $H_{d-1}(\sing_c(<-\ee))$.
\qed
\end{proof}

Let $\locus$ denote the singular locus of $\sing$, that is, the
set $\{ \zz \in \sing : \grad Q(\zz) = \zero \}$.  The point 
$\zz_*$ is a quadratic singularity, thus isolated, and 
we may write
$\locus = \{ \zz_* \} \cup \locus'$ where $\locus'$ is separated
from $\zz_*$ by some positive distance.

\begin{cor}\label{cor:compact}
If the real dimension of $\locus$ is at most $d-2$ then, for 
some $\delta > 0$, the cycles $\{ \gamma_c (\rhat) : 0 < |c| < \delta , 
\rhat \in K \}$ may be chosen so as to be simultaneously supported
by some compact $\bpt$ disjoint from $\locus$.
\end{cor}

\begin{unremark}
In the case where $\locus$ is the singleton $\{ \zz_* \}$ or 
when any additional points $\zz \in \locus$ satisfy $h(\zz) 
\leq h(\zz_*) - \ee$ for all $\rhat \in K$, the proof is just
one line.  This is all our applications presently require, 
however the greater generality (although most likely not best 
possible) may be useful in future work. 
\end{unremark}

\begin{proof}
The first step is to prove that for fixed $\rhat$ we may choose
$\{ \gamma_c (\rhat) : 0 < |c| < \delta \}$ satisfying the
conclusion of Lemma~\ref{lem:split}, all supported on
a fixed compact set $\bpt$ avoiding $\locus$.  It suffices
to avoid $\locus'$ because the condition of being supported
on $\sing_{-\ee}$ immediately implies separation from $\zz_*$.
The construction in Theorem~\ref{th:two cp} produces a single
homotopy for all $c$, which is then intersected with each $\sing_c$.
It follows that the union of the intersection cycles is 
contained in a compact set.  By the dimension assumption,
a small generic perturbation avoids $\locus'$ while still
being separated from $\zz_*$.

Having seen that for fixed $\rhat$ the cycles 
$\{ \gamma_c (\rhat) : 0 < |c| < \delta \}$ may be chosen 
to satisfy the conclusions of Lemma~\ref{lem:split} and
to be supported on a compact set $\bpt (\rhat)$ avoiding $\locus$, the
rest is straightforward.  For each $\rhat$ there is a neighborhood
$\nbd (\rhat) \subseteq K$ such that $\shat \in \nbd$ and 
$h_{\rhat} (\zz) \leq h_{\rhat} (\zz_*) - \ee$ imply 
$h_{\shat} (\zz) \leq h(\zz_*) - \ee/2$.  Thus we may choose 
$\gamma_c (\shat) = \gamma_c (\rhat)$ to be independent of $\shat$
over $\nbd(\rhat)$.  Choosing a finite cover of $K$ by these neighborhoods, 
the union of the corresponding sets $\bpt (\rhat)$ supports the 
cycles $\gamma_c (\rhat)$ for all $c$ and $\rhat$.
\qed\end{proof}

Theorem~\ref{th:two cp} is a rather standard result about pushing 
the intersection class below height $h (\zz_*)$ except in a small
ball about $\zz_*$.  Our proof of Theorem~\ref{th:two cp} 
uses an unspecified torus $\TT'$ with 
polyradius in the descending component $B'$ of 
Definition~\ref{def:descending comp}, and is therefore
not an explicit construction of a chain representing $\TT$,
but is sufficient to prove Theorem~\ref{th:rel zero} and 
Lemma~\ref{lem:split} describing the relative homology
of the pair $(\sing_c , \sing_c (\leq - \ee))$.

Equation~\eqref{eq:T} in Lemma~\ref{lem:lift} is all we need to 
complete the proof of Theorem~\ref{th:actual}.  
However, in Section~\ref{sec:GRZ} we study the asymptotic contributions 
of lower critical points, these being the dominant contributions in
the lacuna setting, when $d$ 
is even and greater than $2k$.  For this purpose we need a more explicit
description of a cycle homologous to $\TT$ at height 
below the critical point: the quadric approximation of $\sing_c$ 
is only good in a neighbourhood of the critical point, 
however finding a torus disjoint from $\sing$
may require traveling further down.
The next lemma finds an explicit cycle homologous to $\TT$, having 
height at most $-\ee$ except for an abitrarily small tube around
a piece of $\sing_{\leq 0}$, in two ways: one when a torus $\TT'$ 
at height $-\ee$ can be chosen disjoint from $\sing$ and a different 
way when $\TT'$ intersects $\sing$.

\begin{lem} \label{lem:lift} 
Choose $\xx \in B$, let 
$\yy = -\ee \xx$, let $\alpha : [0,1] \to \R^d$ be the line
segment from $\xx$ to $\yy$ and define $\TT' = \TT (\yy)$.
\begin{enumerate}[$(i)$]
\item Suppose that $\TT'$ is disjoint from $\sing$, 
as in the proof of Theorem~\ref{th:two cp}.  
Then there exist $\ee , c_*, c' > 0$ 
such that 
\begin{equation} \label{eq:T}
[\TT] = [\leray S_{c}] + [\leray \gamma_{c'}] + [\TT'] 
\end{equation}
for all $|c| < c_*$, where $\gamma_{c'}$ is the cycle in the 
conclusion of Lemma~\ref{lem:split} with $c$ replaced by $c'$.
\item Alternatively, if $\TT'$ is not disjoint from $\sing$
then instead of~\eqref{eq:T} one has
\begin{equation} \label{eq:T+}
[\TT] = [\leray S_{c}] + [(\leray \gamma_{c'})_{\geq - \ee} \# \TT'] 
\end{equation}
where $(\leray \gamma_{c'})_{\geq - \ee} \# \TT'$ is the
connected sum of $(\leray \gamma_{c'})_{\geq - \ee}$ and
$\TT'$ along their common boundary
$(\leray \gamma_{c'})_{= -\ee} = \partial (\TT' \setminus 
\disk \gamma_{c'})$.
\end{enumerate}
\end{lem}

\noindent{\sc Proof:} 
For the compact $\bpt$ described in 
Corollary~\ref{cor:compact}, the intersection of $\sing$ with $\bpt$ 
is smooth.  By Proposition~\ref{pr:tube}, there is a neighborhood
of $\bpt$ in $\sing_* \setminus \locus$ that can be parameterized as 
a $2$-dimensional vector bundle over some compact subset 
$\bpt' \subseteq \sing$.  This bundle is naturally coordinatized 
by the values of $Q$ so that for some small $c'_* > 0$ the 
tubular vicinity around $\sing_\bpt$ can be identified with 
$D' \times \sing_\bpt$ for $\D' := \{c \in \Comp : |c| < c'_* \}$. 
We will denote this vicinity as $\sing_\bpt^{D'}$.

Lemma \ref{lem:split} implies that 
$$ [\TT]=[\leray S_c]+[\leray\gamma_c] +[\TT'] $$
for all small enough $|c|$ (which we may assume from now on 
to be smaller than $c_*<c'_*$). The class $\leray\gamma_c$ 
can be represented by a small tube around a cycle $\gamma_c
\in \sing_c$, which is entirely supported by $\sing_\bpt^{D'}$.
Using the product structure $\sing_\bpt^{D'} \cong D' \times 
\sing_\bpt$ we can identify this tube with a product of a 
small circle (of radius $\rho(c) > 0$) around $c \in D'$ and 
$\gamma_*$, a cycle in the smooth part of $\sing$ obtained by 
projection of $\gamma_c$.  When $c_*$ and $\rho$ are 
sufficiently small, the maximum height of $\gamma_*$ is
$h_* - \ee'$ for some $\ee' > 0$.

There exists a homeomorphism of the annulus $D' - D_{\rho(c)}(c)$ 
fixing its outer boundary and sending the small circle 
$\partial D_{\rho(c)}(c)$ around $c$ into the circle 
of radius $c_*$.  Extend this homeomorphism, fiberwise, 
to all of the tubular vicinity $\sing_\bpt^{D'}$.  Furthermore, 
extend it to the complement of $\sing_\bpt^{D'}$ in such a way 
that it is identity outside of a small vicinity of 
$\sing_\bpt^{D'}$ (and thus near $S_c$ and $\TT, \TT'$). 
Choosing $c_*$ smaller if necessary, and taking $\leray \gamma$
to be the $c_*$-tube around $\gamma_*$ for all $c$ with $|c| < c_*$,
this cycle avoids $\sing_c$ for all $c$ with $|c| < c_*$ and 
has maximum height less than $h_* - \ee$ where $\ee$ is positive
once $c_*$ has been chosen sufficiently small with respect to $\ee'$.
This completes the proof of case~$(i)$. 

\begin{figure}[!ht]
\centering
\includegraphics[height=2in]{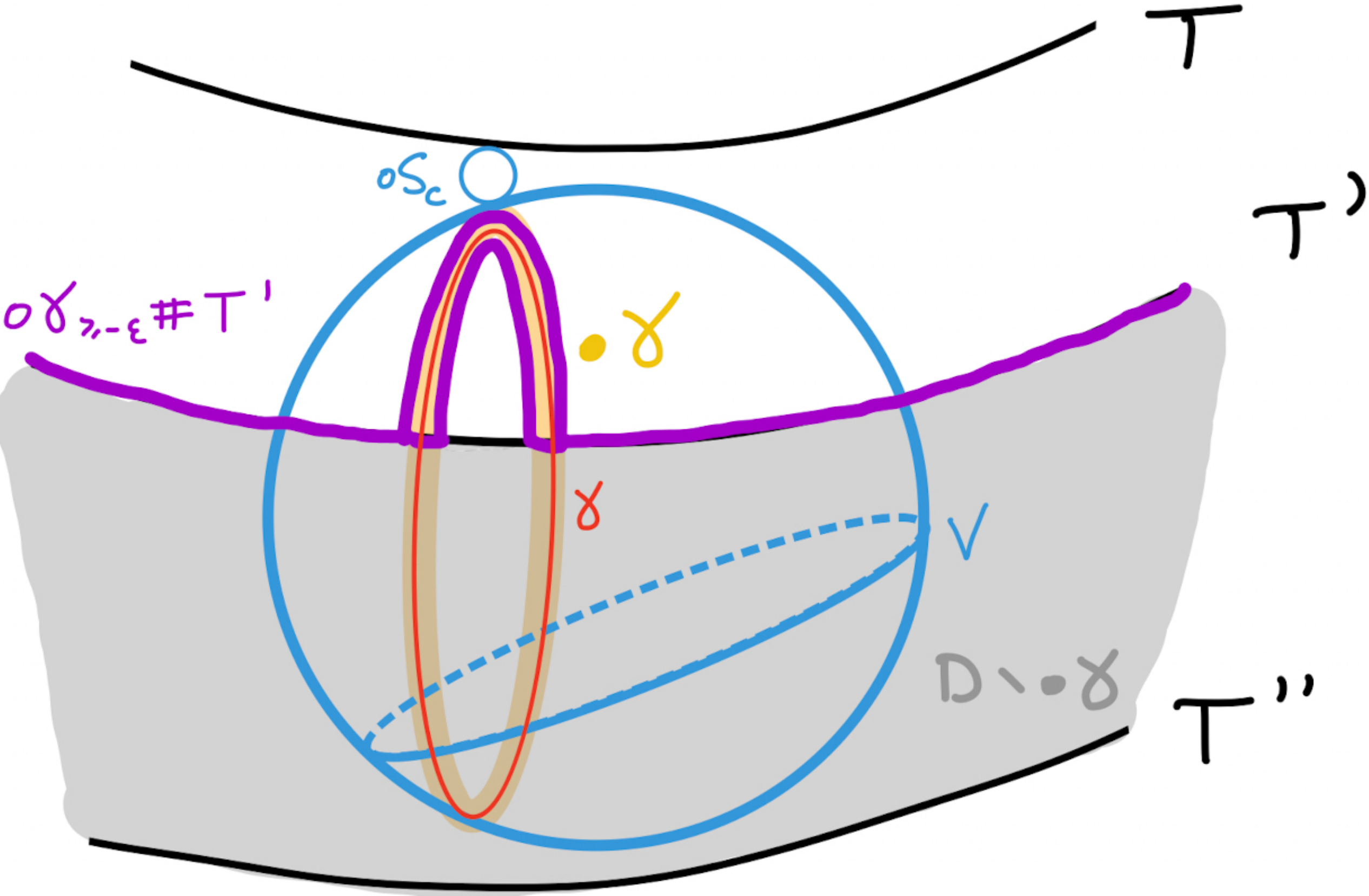}
\caption{Removing a neighborhood of $\sing$ from an expanding torus 
homotopy creates a region in $\M$ whose boundary is a useful cobordism.
The cobordism in~\eqref{eq:T''} is given by
$(\leray \gamma)_{\geq - \ee} \# \TT' + \partial(D\setminus \disk \gamma)
+ \leray S_c$.}
\label{fig:create cycle}
\end{figure}

For case~$(ii)$, let $\alpha: [0,1] \to \R^d$ be as in Theorem~\ref{th:two cp} 
parametrizing the line segment from $\xx$ to some $\yy \in B'$.  Let
$\TT''$ denote the torus with polyradius $\yy$ and let $\TT'$ be the 
torus at height $\ee'$ which is the slice of the homotopy in
Theorem~\ref{th:two cp} for some $t' \in (0,1)$.  Let $\yy'$ be
the corresponding basepoint.

The homotopy swept out by tori with polyradii $\alpha (t)$ intersects $\sing$ 
and defines an intersection cycle $\gamma$.  The homotopy
$\{ \Ll^{-1} \alpha (t) : 0 \leq t \leq 1 \}$ intersects $\sing$, 
yielding an intersection $(d-1)$-chain $\gamma_{\geq \ee}$, 
which is not a cycle.  Its boundary is the $(d-1)$-cycle 
$\gamma_{= -\ee}$ (assuming, without loss of generality, transversal intersections).
These objects are illustrated in Figure~\ref{fig:create cycle}.
Comparing the known expression for $[\TT]$  
\begin{equation} \label{eq:T''}
[\TT] = [\leray S_c] + [\leray \gamma] + [\TT'']
\end{equation}
to the desired expression for $[\TT]$
$$[\TT] = [\leray S_c] + [(\leray \gamma)_{\geq - \ee} \# \TT']$$
using the chain-level identity 
$$(\leray \gamma)_{\geq - \ee} \# \TT' =
   (\leray \gamma)_{\geq - \ee} + (\TT' \setminus \disk (\gamma)) \, ,$$
we see that the difference is represented by the chain
\begin{eqnarray*}
&&   (\leray \gamma)_{\geq - \ee} + (\TT' \setminus \disk (\gamma)) 
   - \leray \gamma - \TT''- \TT'' \\
& = & - (\leray \gamma)_{\leq - \ee} + (\TT' \setminus \disk (\gamma)) 
   - \TT'' \\
& = & \partial (D \setminus \disk (\gamma)) \, ,
\end{eqnarray*}
where $D$ is the $(d+1)$-chain given by the $\Ll$-preimage of
$\{ \alpha (t) : t' \leq t \leq 1 \}$.  Because the difference 
is a boundary, this establishes~\eqref{eq:T+}.	
$\Cox$

\noindent{\sc Proof of Theorem~\ref{th:actual}}
Let $\gamma(\rhat)$ be chosen as in the first conclusion of
Lemma~\ref{lem:lift}, let 
$$ \Gamma(\rhat):=\leray\gamma(\rhat) + \TT(\yy) \, ,$$
which must satisfy condition~$(iii)$ of Theorem~\ref{th:actual},
and choose $\Gamma_c := \leray S_c$.  Conclusion~$(i)$ follows from
the choice of $c_*$ at the end of the proof of Lemma~\ref{lem:lift}.
Conclusion~$(ii)$ is equation~\eqref{eq:T}.  As the compact cycle 
$\Gamma$ is independent of $c$, equation~\eqref{eq:omega_c} follows 
immediately from convergence of $\omega_c$ to $\omega$ on $\Gamma$ 
for each $\rr$.  It remains only to verify~\eqref{eq:gamma_c}.  
 
To prove~\eqref{eq:gamma_c}, choose a local coordinate system in 
which $Q$ is reduced to its quadratic part, and rescale it by $c^{1/2}$ 
(either root will work).  In this coordinate system 
$\uu = \vv + i\ww$, we are integrating over the cycle 
$\leray S_1$, where
$$ S_1=\left\{v_1^2+\sum_{k=2}^d w_k^2=1; w_1=v_2=\cdots=v_d=0\right\} \, . $$
In the new local coordinates $\zz = \zz_* + c^{1/2} \uu \psi(\uu)$ 
with $\psi$ holomorphic and $\psi(0) = 1$, the form 
$\zz^{-\rr} \omega_c$ becomes 
\begin{equation}
\zz^{-\rr} \omega_c = \zz_*^{-\rr-\one} (1 + c^{1/2} \uu / \zz_*)^{-\rr-\one} 
   \frac{P(\zz_* + c^{1/2} \uu \psi(\uu))}{c^k q(\uu)^k} \, c^{d/2} \, d\uu
   = c^{d/2-k} \zz_*^{-\rr-\one} H(\uu,c) \, d\uu \, ,
\end{equation}
where $q$ is the quadric~\eqref{eq:q} and 
$$H(\uu , c ) := (1 + c^{1/2} \uu / \zz_*)^{-\rr-\one} 
   \frac{P(\zz_* + c^{1/2} \uu \psi(\uu))}{q(\uu)^k} \, .$$
The function $H(\uu,c)$ is holomorphic in $\uu$ and bounded on $\leray S_1$ 
uniformly in $c$.  As $c \to 0$, $H(\uu , c) \to P(\zz_*) / q(\uu)^k$ and
the conclusion~\eqref{eq:gamma_c} follows.
$\Cox$

\section{Application to the GRZ function with critical parameter}
\label{sec:GRZ}

Having established the exponential drop, this section extends
Theorem~\ref{th:lacuna} to obtain more precise asymptotics for $a_\rr$.
Most of what follows concentrates on the GRZ example, however 
we first state a result holding more generally in the presence 
of a lacuna. A \emph{critical point at infinity}, formally defined
in~\cite{BMP-morse}, can be viewed as a sequence of singularities going off
to infinity in such a way that the limit of the differential of the 
height function at the points approaches zero. Here, we note only
that there is an effective test for critical points at 
infinity~\cite[Algorithm 1]{BMP-morse} and that 
our GRZ example does not have any.

\begin{thm} \label{th:exp}
Assume the hypotheses of Theorem~\ref{th:lacuna}.  Fix $\rhat$ and 
let $c_1 > c_2$ be the heights of the two highest critical points,
the highest being the quadric singularity.  Suppose, in addition, 
that $Q$ has no critical points at infinity in direction $\rhat$
at any height in $[c_2 , c_1]$.  Then for every $\ee > 0$ there 
is a neighborhood $\nbdhat$ of $\rhat$ such that as $\rr \to \infty$ 
with $\rr / |\rr| \in \nbdhat$, 
$$a_\rr = O \left ( e^{(c_2 + \ee) |\rr|} \right ) \, .$$
\end{thm}

Theorem~\ref{th:exp} is an almost immediate consequence of Theorem~\ref{th:lacuna}
and the following result.

\begin{pr}[\protect{\cite[Theorem~2.4~$(ii)$]{BMP-morse}}] \label{pr:morse}
Let $[a,b]$ be a real interval and suppose that $\sing_*$ has 
no finite or infinite critical points $\zz$ with 
$h_{\rhat} (\zz) \in (a,b]$.  Then for any $\ee > 0$, any chain 
$\Gamma$ of maximum height at most $b$ can be homotopically deformed
into a chain $\Gamma'$ whose maximum height is at most $a + \ee$.
\end{pr}

\noindent{\sc Proof of Theorem}~\ref{th:exp}: 
Apply Proposition~\ref{pr:morse} with $a = c_2$ and $b = c_1$,
resulting in the chain $\Gamma'$.  Applying Theorem~\ref{th:lacuna}
and the homotopy equivalence of $\Gamma$ and $\Gamma'$ in $\M$, 
$$a_\rr = \int_{\Gamma'} \zz^{-\rr} \frac{P}{Q^k} \frac{d\zz}{\zz} + R$$
where $R$ decreases super-exponentially, and in the polynomial case
is in fact zero for all but finitely many $\rr$.  The height
condition on $\Gamma'$ implies that this integral is bounded
above by the volume of $\Gamma'$, multiplied by the maximum value 
of $|F|$ on $\Gamma'$, multiplied by $e^{(c_2 + \ee) |\rr|}$.  
$\Cox$

In the remainder of this section, as in Example~\ref{eg:GRZ}, we let
\begin{equation} \label{eq:GRZ}
F(\zz) := \frac{1}{1 - z_1 - z_2 - z_3 - z_4 + 27 z_1 z_2 z_3 z_4} \, .
\end{equation}
Fix $\rhat$ to be the diagonal direction.  We will prove 
Theorem~\ref{th:zeta} by first computing an estimate for 
$a_\rr$ up to an unknown integer factor $\nint$.  We then use the
theory of D-finite functions and rigorous numerical bounds to 
find the value of $\nint$.  Lastly, we indicate how the value
of $\nint$ could possibly be determined by topological methods.
In order to discuss the sets $\sing (\-\ee)$ relative to different
critical heights, we extend the notation in~\eqref{eq:filtration} via
$$\sing_{\leq t} := \sing \cap \{\zz : h_{\rhat} (\zz) < t \} \, .$$ 

\begin{pr}[\protect{\cite[Proposition~2.9]{BMP-morse}}] \label{pr:morse decomp}
Let $F = P/Q$, $\sing$, the component $B$ and the coefficients $\{ a_\rr \}$ 
be as in Theorem~\ref{th:lacuna}.  Fix $\rhat$ and suppose the
critical values are $c_1 > c_2 > \cdots > c_m$ with $c_1$ being 
the height of the quadric singularity $\zz_*$.  Suppose there are
no critical points at infinity of finite height.  Then there is
a decomposition $\CC = \sum_{j=1}^m \leray \gamma_i$ in $H_d (\sing_*)$
such that for each $j$, $\gamma_j \in \sing_{\leq c_j}$ and is 
either zero in $H_{d-1} (\sing_{\leq c_j})$ or projects to a nonzero
element of $H_{d-1} (\sing_{\leq c_j} , H_{d-1} (\sing_{c_j - \ee}))$.
The decomposition is not unique, but the least $j$ for which $\gamma_j
\neq 0$ and the projection $\pi \gamma_j$ to 
$H_{d-1} (\sing_{\leq c_j} , H_{d-1} (\sing_{c_j - \ee}))$ 
is well defined.
\end{pr}

These cycles represent classes in integer homology, thus giving 
a representation of $a_\rr$ as integer combinations of integrals 
over homology generators of the respective relative homology
groups.  Such integrals are generally computable via saddle-point
integration.  However, determining the integer coefficients appearing 
in this representation can be extremely difficult, related to the 
so-called connection problem for solutions of differential equations.  
Solving the system $Q=0$ and $\nabla Q = \lambda \nabla h_{\rhat}$,
where $\lambda$ is an additional parameter, gives the set of 
critical points.

\begin{pr} \label{pr:crit set}
The critical points of $\sing$ are precisely the points 
$\zz_* := (1/3, 1/3, 1/3, 1/3)$, $\ww := (\zeta , \zeta, 
\zeta, \zeta)$ and $\ww' := \overline{\ww}$, where
$\zeta = (-1 + i \sqrt{2}) / 3$.  
There are no critical points at infinity.
The point $\zz_*$ is a quadric singularity.
$\Cox$
\end{pr}

Let $c_1 = h_{\rhat} (\zz_*) = \log 81$ and $c_2 = h_{\rhat} (\ww) 
= h_{\rhat} (\ww') = \log 9$.  Generators for the rank-2 homology group
$H_3 (\sing_{\leq c_2} , \sing_{\leq c_2 - \ee})$ are given by the 
unstable manifold for downward gradient flows at $\ww$ and $\ww'$
respectively; denote these chains by $\gamma$ and $\gamma'$.  
The conclusion of Theorem~\ref{th:lacuna} in this case is that
$$a_{\rr} = \int_{\Gamma} \zz^{-\rr} F(\zz) \frac{d\zz}{\zz}$$
where $\Gamma \in \sing_{< c_1}$.  By Proposition~\ref{pr:morse decomp},
$\Gamma = \nint \leray \gamma + \nint' \leray \gamma'$ in 
$H_3 (\sing_*)$ for some integers $\nint$ and $\nint'$, which
must be equal because the coefficients are real.  Explicit formulas
in~\cite[Section~9.5]{PW-book} evaluate $\int_\gamma \zz^{-\rr} F d\zz / \zz$,
which, after adding the complex conjugate, lead to the result in
Theorem~\ref{th:zeta} with~3 replaced by $\nint$.  To prove 
Theorem~\ref{th:zeta}, it remains to determine the integer $\nint$.

\subsubsection*{D-finite Asymptotics and Connection Coefficients}

A univariate complex function $f(z)$ is called \emph{D-finite} 
if it satisfies a linear differential equation with polynomial coefficients,
\begin{equation} \label{eq:Dfin} 
p_r(z)f^{(r)}(z) + p_{r-1}(z)f^{(r-1)}(z) + \cdots + p_0(z) f(z) = 0  \, ,
\end{equation}
where $p_r(z) \not\equiv 0$.  We call such a linear differential equation 
with polynomial coefficients a {\bf D-finite} equation.  Our approach 
to determining $\nint$ relies on the fact that the diagonal of 
a rational function is D-finite~\cite{Christol1984,lipshitz-diagonal}, 
and that asymptotics of D-finite function power series coefficients 
can be determined up to constants which can be rigorously approximated 
to large accuracy.  In general it is not possible to determine 
these constants exactly without additional information (in fact, 
there does not even exist a good characterization of what numbers 
appear as such constants) but knowing asymptotics of $a_{n,n,n,n}$ 
up to an \emph{integer} allows us to immediately determine the value 
of $\nint$.

The process of determining an annihilating D-finite equation of the diagonal of a rational function lies in the domain of \emph{creative telescoping}, a well developed area of computer algebra. In particular, there are popular packages in MAGMA~\cite{Lairez2016} and Mathematica~\cite{Koutschan2010b} which take a multivariate rational function and return an annihilating D-finite equation. For the running example of this section, the diagonal $f(z) = \sum_{n \geq 0}a_{n,n,n,n}z^n$ satisfies the linear differential equation
\begin{equation} 
z^2(81z^2+14z+1)f^{(3)}(z) + 3z(162z^2+21z+1)f^{(2)}(z) + (21z+1)(27z+1)f'(z) + 3(27z+1)f(z)=0.
\label{eq:ODE}
\end{equation}
The following standard results on the analysis of D-finite functions 
can be found in 
Flajolet and Sedgewick~\cite[Section VII. 9]{flajolet-sedgewick-anacomb}.
\begin{itemize}
\item The solutions of a D-finite equation form a $\C$-vector space, 
here equal to dimension three. 
\item A solution of~\eqref{eq:ODE} can only have a singularity when 
the leading polynomial coefficient $z^2(81z^2+14z+1)$ vanishes. 
Here the roots are 0, $\zeta^4$, and its algebraic conjugate 
$\overline{\zeta}^4$, where $\zeta$ is the complex number appearing in 
the coordinates of the critical point $c_2$.
\item Equation~\eqref{eq:ODE} is a \emph{Fuchsian} differential equation, 
meaning its solutions have only regular singular points, and its indicial
equation has rational roots.  Because of this, 
at any point $\omega \in \C$, including potentially singularities, 
any solution of~\eqref{eq:ODE} has an expansion of the form 
\begin{equation} (1-z/\omega)^{\alpha} \sum_{j=0}^d {\Big(}
   g_j(1-z/\omega) \log^j(1-z/\omega){\Big)} \label{eq:singExp} 
\end{equation}
in a disk centered at $\omega$ with a line from $\omega$ to the boundary of 
the disk removed, where $\alpha$ is rational and each $g_j$ are analytic. 
At any algebraic point $z=\omega$ there are effective algorithms to 
determine initial terms of the expansion~\eqref{eq:singExp} for a 
basis of the vector space of solutions of~\eqref{eq:ODE}.
\item If $g(z)=\sum_{n \geq 0} c_nz^n$ is a solution of~\eqref{eq:ODE} 
which has no singularity in some disk $|z| < \rho$ except at a point $z=\omega$,
and $g(z)$ has an expansion~\eqref{eq:singExp} in a slit disk near $\omega$
(a disk centered at $\omega$ minus a ray from the center to account for a branch 
cut) then asymptotics of 
$c_n$ are determined by adding asymptotic contributions of the terms 
in~\eqref{eq:singExp}. In particular, a term of the form 
$C(1-z/\omega)^{\alpha}\log^r(1-z/\omega)$ with $\alpha \notin\N$ 
gives an asymptotic contribution of 
$\omega^{-n}n^{-\alpha-1}\log^r(n)\frac{C}{\Gamma(-\alpha)}$ to $c_n$. 
Furthermore, if $g(z)$ has a finite number of singularities in a disk
and each has the above form, then one can simply add the asymptotic 
contributions coming 
from each point in the disk to determine asymptotics of $c_n$. 
\end{itemize}

These results, combined with rigorous algorithms for numerical analytic continuation of D-finite functions, allow us to rigorously determine asymptotics. For our example, the Sage ore\_algebra package~\cite{KauersJaroschekJohansson2015} computes a basis of solutions to~\eqref{eq:ODE} whose expansions at the origin begin
\begin{align*}
a_1(z) &= \log(z)^2  \left(\frac{1}{2} - \frac{3z}{2} + \frac{9z^2}{2} + \cdots\right) + \log(z){\Big(}-4z + 18z^2+\cdots {\Big)} + {\Big(}8z^2 - 48z^3 + \cdots{\Big)} \\[+2mm]
a_2(z) &= \log(z){\Big(}1 - 3z +9z^2 + \cdots{\Big)} + {\Big(}- 4z + 18z^2 + \cdots {\Big)} \\[+2mm]
a_3(z) &= 1 - 3z + 9z^2 + \cdots 
\end{align*}
and a basis of solutions to~\eqref{eq:ODE} whose expansions at $z=\zeta^4$ begin
\begin{align*}
b_1(z) &= 1 + \left( \frac{13}{2} + \frac{43\sqrt{2}}{4}i \right)(z-\zeta^4)^2 + \left( \frac{8165}{48} + \frac{943\sqrt{2}}{30}i \right)(z-\zeta^4)^3 + \cdots \\[+2mm]
b_2(z) &= \sqrt{z - \zeta^4} + \left( \frac{13}{3} - \frac{365\sqrt{2}}{96}i \right)(z-\zeta^4)^{3/2} - \left( \frac{7071}{1024} - \frac{1041\sqrt{2}}{32}i \right)(z-\zeta^4)^{5/2} + \cdots \\[+2mm]
b_3(z) &= (z-\zeta^4) + \left( \frac{17}{3} - \frac{31\sqrt{2}}{6}i \right)(z-\zeta^4)^2 - \left( \frac{1013}{72} + \frac{1805\sqrt{2}}{36}i \right)(z-\zeta^4)^3  + \cdots.
\end{align*}

Because we can compute the power series coefficients of the diagonal generating function $f(z)$ at the origin, we can represent $f(z)$ in this $a_j(z)$ basis. In fact, because $f(z)$ is analytic at the origin it must be a multiple of $a_3(z)$ and examining constant terms shows that $a_3(z)=f(z)$. Because the coefficients of $f(z)$ grow, it must admit a singularity at $z=\zeta^4$ or $z=\overline{\zeta}^4$ (in fact, we can deduce it will have a singularity at both because we already know its dominant asymptotic behaviour). If we can determine $f(z)$ in terms of the $b_j(z)$ basis then we will know its expansion in a neighbourhood of the origin, and therefore be able to determine asymptotics of its coefficients. Thus, we need to solve a \emph{connection problem}, representing a function given by a basis specified by local information at one point in terms of a basis specified by local information at another point. 

To do this it is sufficient to determine the change of basis matrix converting from the $a_j(z)$ basis into the $b_j(z)$ basis. Using algorithms going back to the Chudnovsky brothers~\cite{ChudnovskyChudnovsky1986,ChudnovskyChudnovsky1987} and van der Hoeven~\cite{Hoeven2001}, and recently improved and implemented by Mezarobba~\cite{Mezzarobba2016,Mezzarobba2019}, we can compute this change of basis matrix numerically to any specified precision. The key is to use numeric analytic continuation to evaluate the $a_j(z)$ and $b_j(z)$ to sufficiently high precision near a fixed value of $z$. Such evaluations can be done using the series expansions around each point (which can be computed efficiently) and rigorous bounds on the error of series truncation~\cite{MezzarobbaSalvy2010}. 

In particular, computing the change of basis matrix in this example using the Sage implementation of Mezzarobba gives
\[ f(z) = a_3(z) = C_1b_1(z) + C_2b_2(z) + C_3b_3(z), \]
where $C_1,C_2,$ and $C_3$ are constants which can be rigorously computed to 1000 decimal places in under 10 seconds on a modern laptop. As $b_2(z)$ is the only element of the $b_j(z)$ basis which is singular at $z=c_2$, the dominant singular term in the expansion of $f(z)$ near $z=c_2$ is
\[ C_2 \sqrt{z-\zeta^4} = -{\Big(}\left(3.5933098558743233\ldots{\Big)} + i{\Big(}0.38132214909311386\ldots{\Big)}\right)\sqrt{z-\zeta^4}. \]
Thus, $f(z)$ has a singularity at $z=\zeta^4$ and the asymptotic contribution of this singularity to $a_{n,n,n,n}$ is
\[ \Psi_1(n) := \frac{\left(4i\sqrt{2}-7 \right)^n}{n^{3/2}} \, \frac{{\Big(}\left(0.543449606382202\ldots{\Big)} + i{\Big(}0.259547320313100\ldots{\Big)}\right)}{\sqrt{\pi}} + O(9^nn^{-5/2}). \]
Repeating the same analysis at the point $z=\overline{\zeta}^4$ gives an asymptotic contribution
\[ \Psi_2(n) := \frac{\left(4i\sqrt{2}+7 \right)^n}{n^{3/2}} \, \frac{{\Big(}\left(0.543449606382202\ldots{\Big)} - i{\Big(}0.259547320313100\ldots{\Big)}\right)}{\sqrt{\pi}} + O(9^nn^{-5/2}), \]
so that $a_{n,n,n,n}$ has the asymptotic expansion $a_{n,n,n,n} = \Psi_1(n) + \Psi_2(n)$.

Comparing this expansion, with numerical coefficients known to 
1000 decimal places, to the expansion in~\eqref{eq:anLambda} 
which has constants that are unknown but restricted to be integers, 
proves that $\lambda_2=\lambda_3$ are integers equal to $2.99\dots$ 
up to almost 1000 decimal places (almost 1000 decimal places more 
than needed to make this conclusion), meaning $\nint = 3$.
This finishes the proof of Theorem~\ref{th:zeta}. 
$\Cox$

\section{Concluding remarks}

\subsection*{Explaining the multiplicity}

We have seen that the integral over $\CC(c_2)$ and $\CC(c_3)$ appear 
in the Cauchy integral representation of $a_{n,n,n,n}$ with a multiplicity 
of~3.  Expanding a torus past a smooth critical point leads to a 
coefficient of~1 when the critical point is a height maximum along
the imaginary fiber and zero when it is a height minimum along this 
fiber.  Evidently, when deforming the Cauchy domain of integration 
past the highest critical point $c_1=(1/3,1/3,1/3,1/3)$, the 
resulting chain $\Gamma$ lying just below this height is not
like a simple torus and instead, under gradient flow, has 
multiplicity~3 in the local homology basis at the diagonal
points $\zeta$ and $\overline{\zeta}$.  

\begin{problem}
Give a direct demonstration of these coefficients being~3.
\end{problem}

Our best explanation at present is this.  If $W$ is a smooth
algebraic hypersurface, Morse theory gives us a basis for 
$H_{d-1} (W)$ consisting of the unsable manifolds for downward
gradient flow at each critical point.  The stable manifolds
at each critical point are an upper tringular dual to this
via the intersection pairing.  The original torus of integration
is a tube over a torus $T_0$ in $\sing$.  If $\sing$ were smooth, 
we would be trying to show that the stable manifold at $\ww$ in 
$\sing_*$ intersects $T_0$ with signed multiplicity~$\pm 3$, 
where $\ww = (\zeta, \zeta, \zeta, \zeta)$.  
This is probably not true in the smooth varieties $\sing_c$. 
However, as $c \to 0$, part of the stable manifold at $\ww$
gets drawn toward $\zz_* = (1/3, 1/3, 1/3, 1/3)$.  Therefore,
in the limit, we need to check how many total signed paths
in the gradient field ascend from $\ww$ to $\zz_*$.

By the symmetry, we expect to find these paths along the three partial 
diagonals: $\{ x=y, z=w \}$, $\{ x=z, y=w \}$ and $\{ x=w, y=z \}$. 
Solving for gradient ascents on any one of these yields three that
go to $\zz$ rather than to the coordinate planes.  If these all had 
the same sign, the multiplicity would be~9 rather than~3, therefore, 
in any one partial diagonal, the three paths are two of one sign and
one of the other.  It remains to show that the signs are as predicted,
that these are the only paths going from $\ww$ to $\zz_*$, and to
rigorize passage from the smooth case to the limit as $c \to 0$.

\subsection*{Computational Morse theory}

One of the central problems in ACSV is effective computation of
coefficients in integer homology.  Specifically, the class
$[T] \in H_d (\M)$ must be resolved as an integer combination
of classes $\leray \sigma$ where $\sigma \in H_{d-1} (\sing_*)$
projects to a homology generator for one of the attachment
pairs $H_{d-1} (\sing_{\leq c} , \sing_{\leq c - \ee})$ near
a critical point with critical value $c$.  What is known
is nonconstructive.  There is a highest critical value $c$ where
$[T]$ has nonzero homology in the attachment pair.  The
projection of $[T]$ to $H_{d-1} (\sing_{\leq c} , \sing_{\leq c - \ee})$ 
is well defined.  If this relative homology element is the projection
of an absolute homology element $\sigma \in H_{d-1} (\sing_{\leq c} 
\setminus \sing_{\leq c - \ee})$ then there is no {\em Stokes phenomenon},
meaning one can replace $[T]$ by $[T] - \sigma$ and continue to the
next lower attachment pair where $[T] - \sigma$ projects to a nonzero
homology element.

The data for this problem is algebraic.  Therefore, one might hope for 
an algebraic solution, which can be found via computer algebra without
resorting to numerical methods, rigorous or otherwise.  At present,
however, we have only heuristic geometric arguments.

\begin{problem}
Given an integer polynomial and rational $\rhat$, algebraically 
compute the highest critical points $\zz$ for which the projection of $[T]$
to the attachment pair is nonzero.  Then compute these integer 
coefficients.  Also determine whether $T$ is homologous to a local cycle, 
and in the case that it is, find a way to continue the computation 
to the next lower critical point.
\end{problem}

\subsection*{Combining computation and topology}

One of the main achievements of the present paper is the preceding
chain of reasoning that combines topological methods with computer
algebra.  Computer algebra methods give asymptotic formulae for
the diagonal coefficients which includes an unknown constant, computable
up to an arbitrarily small (rigorous) error term.  These methods 
say nothing about the behavior of coefficients in a neighborhood of
the diagonal.  Topological methods show that in a neighborhood of
the diagonal, coefficients are given by an asymptotic formula which
is the sum of algebraic quantities up to unknown integer factors.
This method on its own cannot identify the correct asymptotics
without further geometric methods that have, thus far, eluded us.
Combining the two analyses determines the integer factors, leading
to rigorous asymptotics throughout an open cone containing the 
diagonal direction.

\bibliographystyle{alpha} 
\bibliography{bibl}

\end{document}